\documentclass[11pt, reqno]{amsart}

\usepackage[utf8]{inputenc}
\usepackage[T1]{fontenc}
\usepackage{lmodern}
\usepackage[margin=2.5cm]{geometry}
\usepackage{microtype}

\usepackage{amsmath,amssymb,amsthm,mathtools}
\mathtoolsset{showonlyrefs}
\usepackage{braket}
\usepackage{tikz-cd}

\usepackage{graphicx}
\usepackage{booktabs}
\usepackage{xcolor}

\usepackage[colorlinks=true,
            linkcolor=blue,
            citecolor=red,
            urlcolor=blue,
            pdfborder={0 0 0}]{hyperref}

\newtheorem{theorem}{Theorem}[section]
\newtheorem{proposition}[theorem]{Proposition}

\theoremstyle{definition}

\theoremstyle{remark}
\newtheorem{remark}[theorem]{Remark}
\newcommand{\Z}{\mathbb{Z}}
\newcommand{\C}{\mathbb{C}}
\newcommand{\HH}{\mathcal{H}}

\newcommand{\cO}{\mathcal{O}}
\newcommand{\cA}{\mathcal{A}}
\newcommand{\cF}{\mathcal{F}}

\newcommand{\F}{\mathbb{F}}
\newcommand{\qbinom}[2]{\genfrac{[}{]}{0pt}{}{#1}{#2}}

\DeclareMathOperator{\Map}{Map}
\DeclareMathOperator{\Aut}{Aut}

\DeclareMathOperator{\Sp}{Sp}

\DeclareMathOperator{\Ad}{Ad}
\DeclareMathOperator{\GL}{GL}
\DeclareMathOperator{\Irr}{Irr}
\DeclareMathOperator{\Ind}{Ind}

\DeclareMathOperator{\SL}{SL}
\DeclareMathOperator{\Hom}{Hom}

\DeclareMathOperator{\ASp}{ASp}
\DeclareMathOperator{\PsSp}{PsSp}

\DeclareMathOperator{\Rep}{Rep}
\DeclareMathOperator{\id}{id}

\DeclareMathOperator{\Stab}{Stab}
\DeclareMathOperator{\End}{End}
\DeclareMathOperator{\PGL}{PGL}
\DeclareMathOperator{\Tr}{Tr}

\DeclareMathOperator{\Vect}{Vect}

\title[Representation theory of projective Clifford groups]{Representation theory of projective Clifford groups via isocategoricality}

\author{C\'esar Galindo}
\address{Departamento de Matem\'aticas, Universidad de los Andes, Bogot\'a, Colombia}
\email{cn.galindo1116@uniandes.edu.co}

\subjclass[2020]{Primary 20C15, 20C25; Secondary 18M20}
\keywords{Clifford group, affine symplectic group, isocategorical groups,
character table, Weil representation, tensor-power commutant, unitary design}

\date{}

\begin{document}

\begin{abstract}
We prove that the projective Clifford group \(C(A)\) associated with a finite
abelian group \(A\) is isocategorical to the affine symplectic group
\(\ASp(A)=\Sp(V_A)\ltimes\widehat{V_A}\), where
\(V_A=A\oplus\widehat{A}\). This explicit isomorphism of tensor categories gives
common parametrizations of their irreducible characters and conjugacy classes,
under which their character tables agree entry by entry. It also gives an
explicit basis for every tensor-power commutant of the projective Weil
representation, indexed by symplectic orbits of tuples in \(V_A\) whose entries
sum to zero. For cyclic \(A\), we classify these orbits by oriented
two-generated submodules and obtain exact commutant dimensions in every tensor
power.
\end{abstract}

\maketitle

\section{Introduction}\label{sec:intro}

Let \(A\) be a finite abelian group and let \(\HH=\C[A]\) be the Hilbert
space with basis \(\{\ket{b}:b\in A\}\). For
\(a\in A\) and \(\chi\in\widehat{A}=\Hom(A,U(1))\), define
\[
X_a\ket{b}=\ket{a+b},
\qquad
Z_\chi\ket{b}=\chi(b)\ket{b}.
\]
For \(v=(a,\chi)\), write \(W_v:=X_aZ_\chi\). These are the Pauli, or Weyl,
operators attached to \(A\), and their labels form
\(V_A:=A\oplus\widehat{A}\). The associated Heisenberg and projective
Clifford groups are
\[
\begin{aligned}
H(A)&:=\{zX_aZ_\chi:z\in U(1),\ a\in A,\ \chi\in\widehat{A}\},\\
C(A)&:=N_{U(\HH)}(H(A))\big/\bigl(U(1)\cdot I\bigr).
\end{aligned}
\]
Thus Clifford operators are precisely the unitaries whose conjugation
permutes the Weyl operators up to scalar multiples.
When \(A=(\Z/d\Z)^n\), this is the projective Clifford group on \(n\)
qudits, with \(d=2\) giving the projective \(n\)-qubit Clifford group. These
groups occur throughout the stabilizer formalism, fault-tolerant quantum
computation, randomized benchmarking, and the study of finite collections of
unitaries that reproduce low-order moments of Haar measure
\cite{Gottesman1997,Gottesman1998,AaronsonGottesman2004,BravyiKitaev2005,
DankertCleveEmersonLivine2009,KnillBenchmarking2008,
MagesanGambettaEmerson2011}. Allowing arbitrary finite \(A\) gives a uniform
description of cyclic and multi-qudit systems in terms of finite abelian groups
\cite{Appleby2005,Gross2006,HostensDehaeneDeMoor2005}. In the qudit case, this
construction is also the finite analogue of the Weil representation
\cite{Mackey1958,Weil1964}.

The multiplication of Weyl operators is governed by the bicharacter
\[
\beta_A\bigl((a,\chi),(b,\psi)\bigr):=\chi(b),
\qquad
W_vW_w=\beta_A(v,w)W_{v+w}.
\]
Its antisymmetrization
\[
\omega_A(v,w):=\frac{\beta_A(v,w)}{\beta_A(w,v)}
\]
is a nondegenerate alternating bicharacter on \(V_A\). The quotient
\(H(A)/(U(1)\cdot I)\) identifies with \(V_A\) endowed with this symplectic
form.
Conjugation by Clifford operators preserves \(\omega_A\). As explained in
Subsection~\ref{subsec:clifford-extension}, the resulting quotient map gives
the exact sequence
\begin{equation}\label{eq:intro-extension}
1\longrightarrow V_A\longrightarrow C(A)
\longrightarrow \Sp(V_A)\longrightarrow 1.
\end{equation}
We simply call \(C(A)\) the Clifford group. The extension
\eqref{eq:intro-extension} splits exactly when \(4\nmid |A|\)
\cite[Theorem~6.1]{GalindoClifford2026}; in particular, it cannot in general be
treated as a semidirect product. Despite this failure to split,
\(\Rep(C(A))\) is tensor equivalent to the representation category of a split
affine symplectic group.

Let \(\ASp(A):=\Sp(V_A)\ltimes\widehat{V_A}\) be the affine symplectic group.
Theorem~\ref{thm:tensor-equivalence} constructs an explicit isomorphism of
tensor categories
\begin{equation}\label{eq:intro-main-equivalence}
\cF:\Rep(\ASp(A))
\xrightarrow{\ \sim\ }
\Rep(C(A)).
\end{equation}
The normal abelian subgroup \(\widehat{V_A}\subseteq\ASp(A)\) gives every
\(\ASp(A)\)-module a natural \(V_A\)-weight decomposition
\(U=\bigoplus_{v\in V_A}U_v\). The functor \(\cF\) preserves this vector space
and grading. On homogeneous vectors its tensor constraint is
\[
J_{U,W}(\xi_v\otimes\eta_w)
=
\beta_A(v,w)\,\xi_v\otimes\eta_w.
\]
Theorem~\ref{thm:symmetric-equivalence} shows that \(\cF\) becomes symmetric
after twisting the usual symmetry on \(\Rep(\ASp(A))\) by \(\omega_A\). Thus
\(C(A)\) and \(\ASp(A)\) are
isocategorical in the sense of Davydov and Etingof--Gelaki
\cite{Davydov2001,EtingofGelaki,GalindoIsocategoricalWeil}. This already
implies that their ordinary character tables agree up to relabeling; one of
our goals is to make that relabeling explicit.

The remainder of the paper develops two families of consequences of this
construction.

\smallskip
\noindent\emph{Representations and character theory.}
By Theorem~\ref{thm:clifford-classification}, Mackey's little-group method
applied to the affine symplectic group \(\ASp(A)\), and transported by \(\cF\),
parametrizes the irreducible representations of \(C(A)\) by a symplectic orbit
in \(V_A\) and an ordinary irreducible character of its stabilizer.
Theorem~\ref{thm:inertia-character-extension}
strengthens this description: every character of the kernel \(V_A\) extends
to its inertia subgroup, namely its stabilizer inside \(C(A)\). Consequently,
Mackey's construction uses ordinary representations of the stabilizer rather
than projective ones, even when \eqref{eq:intro-extension} does not split. For
elementary abelian \(2\)-groups, this proves the Basheer--Moori conjecture in
Clifford--Fischer theory \cite[Section~5.2]{BasheerMooriSurvey2015}; see
Theorem~\ref{thm:basheer-moori-conjecture}.

For multiqubit Clifford groups, Helsen, Wallman, and Wehner determined the
irreducible constituents of the two-copy representation, while Lee, Yu, Peng,
and Lai computed the full character tables for \(n=1,2,3\) using presentations
and computer algebra and applied them to tensor-power decompositions
\cite{HelsenWallmanWehner2018,LeeYuPengLai2025}. Our results instead give
uniform row and column parameters for Clifford groups associated with arbitrary
finite abelian groups.

The conjugacy classes admit a description independent of any choice of section
of \eqref{eq:intro-extension}. Over a symplectic element \(T\), they are
controlled by centralizer orbits of characters of the fixed-point subgroup
\(\ker(T-\id)\); see
Theorem~\ref{thm:intrinsic-clifford-conjugacy}. These are also the column
parameters for \(\ASp(A)\). Combined with the common row parameters,
Theorem~\ref{thm:common-character-table} gives an entry-by-entry identification
of the two character tables using only symplectic orbits, fixed-point
subgroups, and ordinary character tables of stabilizers. We work this out for
\(A=\Z/2^k\Z\) and \(A=(\Z/2\Z)^m\) in
Subsections~\ref{subsec:cyclic-2power-example} and
\ref{subsec:elementary-2group-example}, respectively.

\smallskip
\noindent\emph{The Weyl algebra and tensor-power commutants.}
The Weyl operators span an algebra \(\cA\simeq\End_\C(\HH)\), identified with
the twisted group algebra \(\C_{\beta_A}[V_A]\).
Theorem~\ref{thm:weyl-transport} shows that the same tensor structure relates
\(\cA\) to the ordinary group algebra \(\C[V_A]\). The projective Weil
representation of \(C(A)\) on \(\HH\) is constructed in
Subsection~\ref{subsec:twisted-weyl-algebra}. Let \(\mathfrak D_t(A)\) be the
commutant of its \(t\)-th tensor power on \(\HH^{\otimes t}\). If
\[
Z_t(A):=\{(v_1,\ldots,v_t)\in V_A^t:v_1+\cdots+v_t=0\},
\]
then Theorem~\ref{thm:tensor-power-commutant} constructs an orthogonal orbit
basis and gives
\[
\dim\mathfrak D_t(A)=|\Sp(V_A)\backslash Z_t(A)|.
\]
Proposition~\ref{prop:tensor-commutant-product} gives an explicit multiplication
rule. In the usual prime-dimensional qudit setting, related tensor-power
commutants have been studied in
\cite{GrossNezamiWalter2021,MontealegreMoraGross2025,BittelEtAl2026}; the
present description is uniform for arbitrary finite abelian \(A\).

When \(A=\Z/N\Z\) is cyclic, we evaluate the orbit problem completely.
Theorem~\ref{thm:oriented-submodule-classification} classifies the
relevant orbits by two-generated submodules equipped with a generator of their
second exterior power. Combining this classification with Birkhoff's
subgroup-counting formula gives the exact dimension formula of
Theorem~\ref{thm:cyclic-commutant-dimension}; in particular,
\[
\dim\mathfrak D_2(\Z/N\Z)=\tau(N),
\qquad
\dim\mathfrak D_3(\Z/p^r\Z)
=p^{r-1}\bigl((r+1)p+r\bigr).
\]
Here \(\tau(N)\) denotes the number of positive divisors of \(N\). Finally,
trace duality gives a vector-space
isomorphism between the adjoint-action commutant and \(\mathfrak D_{2t}(A)\);
see Proposition~\ref{prop:trace-doubling}.

The paper is organized as follows. Section~\ref{sec:prelim} reviews the
Heisenberg group associated with \(A\), the symplectic structure on \(V_A\),
and the description of \(C(A)\) by pairs \((T,\lambda)\). It also introduces
the affine symplectic group \(\ASp(A)\). Section~
\ref{sec:symmetric-tensor-equivalence} constructs the tensor isomorphism
between \(\Rep(\ASp(A))\) and \(\Rep(C(A))\) and determines the symmetry
transported by this isomorphism. Section~\ref{sec:clifford-reps} applies the
tensor isomorphism to irreducible representations and conjugacy classes,
proves that every kernel character extends to its inertia subgroup, and
derives the common character formula. It also treats the cyclic \(2\)-power
and elementary abelian \(2\)-group cases, including the Basheer--Moori
conjecture. Finally, Section~\ref{sec:weil-algebra} studies the Weyl algebra
and the projective Weil representation, constructs orbit bases for their
tensor-power commutants, solves the orbit problem when \(A\) is cyclic, and
relates tensor-power and adjoint-action commutants by trace duality.

\subsection*{Acknowledgements}
The author was partially supported by Grant INV-2025-213-3452 from the School
of Science of Universidad de los Andes. The author collaborated with ChatGPT
5.6 during the preparation of this revised version of the paper.
\section{Preliminaries}\label{sec:prelim}

Here we fix the notation used throughout the paper. Finite abelian groups are
written additively and their
characters multiplicatively. Let \(A\) be a finite abelian group,
\[
\widehat{A}:=\Hom(A,U(1)),
\]
and
\[
V_A:=A\oplus\widehat{A}.
\]
All representations are finite-dimensional over \(\C\), unless explicitly
stated otherwise.

\subsection{The Heisenberg subgroup and the symplectic structure}\label{subsec:heisenberg-symplectic}

Let
\[
\HH=\C[A]
\]
with basis \(\{\ket{a}:a\in A\}\), equipped with the Hermitian form for which
this basis is orthonormal. For \(a\in A\), \(\chi\in\widehat{A}\), and
\(u=(a,\chi)\in V_A\), put
\[
X_a\ket{b}=\ket{a+b},\qquad
Z_\chi\ket{b}=\chi(b)\ket{b},\qquad
W_u:=X_aZ_\chi.
\]
The associated Heisenberg bicharacter is
\begin{equation}\label{eq:heisenberg-cocycle}
\beta_A\bigl((a,\chi),(b,\psi)\bigr):=\chi(b),
\end{equation}
and therefore
\[
W_uW_v=\beta_A(u,v)\,W_{u+v}.
\]
The Heisenberg subgroup attached to \(A\) is
\[
H(A):=\{zW_u\mid z\in U(1),\ u\in V_A\}\subset U(\HH).
\]
It is a central extension of \(V_A\) by \(U(1)\), with projective quotient
canonically isomorphic to \(V_A\):
\[
H(A)/(U(1)\cdot I)\cong V_A.
\]
The antisymmetrization of \(\beta_A\) is the canonical symplectic form
\begin{equation}\label{eq:symplectic-form}
\omega_A(u,v):=\beta_A(u,v)\,\beta_A(v,u)^{-1}
=\chi(b)\psi(a)^{-1},
\qquad
u=(a,\chi),\ v=(b,\psi).
\end{equation}
It is alternating and nondegenerate. The commutator in \(H(A)\) is controlled by
\[
W_uW_v=\omega_A(u,v)\,W_vW_u.
\]
Thus
\begin{equation}\label{eq:symplectic-identification}
\vartheta:V_A\xrightarrow{\sim}\widehat{V_A},
\qquad
\vartheta_v(u):=\omega_A(v,u),
\end{equation}
is an isomorphism; we use it to pass between elements of \(V_A\) and characters
of \(V_A\). For \(T\in\Sp(V_A)\), it satisfies
\[
\vartheta_{Tv}=\vartheta_v\circ T^{-1}.
\]

We use two different dualities. Pontryagin duality is the evaluation
isomorphism
\[
\iota:V_A\xrightarrow{\sim}\widehat{\widehat{V_A}},
\qquad
\iota_v(\chi):=\chi(v),
\]
and this is the convention used to label the weights of
\(\widehat{V_A}\)-representations. The map \(\vartheta\), instead, uses the
symplectic form and identifies \(V_A\) with \(\widehat{V_A}\); it is the
convention used to write the kernel of the Clifford extension.

We write
\begin{equation}\label{eq:def-SpVA}
\Sp(V_A):=\{T\in\Aut(V_A)\mid \omega_A(Tu,Tv)=\omega_A(u,v)\text{ for all }u,v\in V_A\},
\end{equation}
and the projective Clifford group is
\begin{equation}\label{eq:def-CA}
C(A):=N_{U(\HH)}(H(A))\big/\,U(1)\cdot I.
\end{equation}

\subsection{The pair description of the Clifford extension}\label{subsec:clifford-extension}

The Clifford group admits a description by pairs. Define the affine
pseudosymplectic group attached to \(\beta_A\) by
\begin{equation}\label{eq:clifford-pair-description}
\PsSp_{\beta_A}(V_A):=
\left\{(T,\lambda)\in\Aut(V_A)\times\Map(V_A,U(1))\;\middle|\;
\frac{\lambda(u+v)}{\lambda(u)\lambda(v)}
=
\frac{\beta_A(Tu,Tv)}{\beta_A(u,v)}
\text{ for all }u,v\in V_A
\right\},
\end{equation}
with product
\[
(T,\lambda)\cdot(S,\mu)=(TS,\lambda^S\mu),\qquad (\lambda^S\mu)(u):=\lambda(Su)\mu(u).
\]
Putting \(u=v=0\) in \eqref{eq:clifford-pair-description} gives
\(\lambda(0)=1\). Moreover, two functions lying over the same \(T\) differ by
a character of \(V_A\).
The displayed condition already forces \(T\in\Sp(V_A)\). Indeed, the left
side is symmetric in \(u\) and \(v\), hence
\[
\frac{\beta_A(Tu,Tv)}{\beta_A(u,v)}
=
\frac{\beta_A(Tv,Tu)}{\beta_A(v,u)}.
\]
Rearranging gives \(\omega_A(Tu,Tv)=\omega_A(u,v)\). Thus the projection to the
first coordinate has image in \(\Sp(V_A)\).

By \cite[Theorem~2.2]{GalindoClifford2026}, the assignment
\[
(T,\lambda)\longmapsto[U],
\]
where \(U\) is any unitary satisfying
\[
UW_uU^\dagger=\lambda(u)W_{Tu}\qquad (u\in V_A),
\]
is an isomorphism
\[
\PsSp_{\beta_A}(V_A)\xrightarrow{\sim} C(A).
\]
Theorem~2.2 of \cite{GalindoClifford2026} also gives surjectivity of the projection to the first coordinate
onto \(\Sp(V_A)\). From now on we identify \(\PsSp_{\beta_A}(V_A)\) with
\(C(A)\) through this isomorphism and continue to write \(C(A)\) for the pair
group. Thus elements of \(C(A)\) are written as pairs \((T,\lambda)\), where
\(\lambda\) need not be a character.

The projection to the first coordinate has kernel
\[
D_C:=\{(\id,\chi):\chi\in\widehat{V_A}\}\cong\widehat{V_A}.
\]
Using \(\vartheta\) to write this kernel as \(V_A\), put
\[
K_v:=(\id,\vartheta_v)
\qquad (v\in V_A).
\]
This is the class of the Weyl operator \(W_v\), since
\[
W_vW_uW_v^{-1}=\omega_A(v,u)W_u.
\]
Then the group \(C(A)\) fits into the exact sequence
\begin{equation}\label{eq:extension}
1\longrightarrow V_A\xrightarrow{\,v\mapsto K_v\,}
C(A)\xrightarrow{\pi}\Sp(V_A)\longrightarrow 1.
\end{equation}
The quotient action on this kernel is the natural action on \(V_A\):
\begin{equation}\label{eq:kernel-equivariance}
(T,\lambda)K_v=K_{Tv}(T,\lambda).
\end{equation}

Choose a normalized set-theoretic section
\[
s:\Sp(V_A)\to C(A),
\qquad
T\mapsto(T,\lambda_T^{(s)}),
\qquad
s(\id)=(\id,1).
\]
Define the factor set \(a_s:\Sp(V_A)\times\Sp(V_A)\to V_A\) by
\begin{equation}\label{eq:obstruction-def}
s(T)s(S)=K_{a_s(T,S)}s(TS).
\end{equation}
In terms of the functions \(\lambda_T^{(s)}\), the defining relation reads
\begin{equation}\label{eq:obstruction-phase-formula}
\omega_A\bigl(a_s(T,S),TSu\bigr)
=
\frac{\lambda_T^{(s)}(Su)\lambda_S^{(s)}(u)}
{\lambda_{TS}^{(s)}(u)}
\qquad (u\in V_A).
\end{equation}
Associativity and \eqref{eq:kernel-equivariance} give
\[
a_s(T,S)+a_s(TS,R)
=
T a_s(S,R)+a_s(T,SR).
\]
Thus \(a_s\) is a normalized \(2\)-cocycle for the natural
\(\Sp(V_A)\)-action on \(V_A\). Its cohomology class is independent of the
normalized section and is the extension class of \eqref{eq:extension}; we
denote it by
\[
[\cO_A]\in H^2(\Sp(V_A),V_A).
\]
The extension \eqref{eq:extension} splits if and only if \([\cO_A]=0\).

The split group used in the tensor isomorphism is the affine symplectic group
\[
\ASp(A):=\Sp(V_A)\ltimes \widehat{V_A}
=
\{(T,\chi)\mid T\in \Sp(V_A),\ \chi\in \widehat{V_A}\},
\]
with multiplication
\[
(T,\chi)(S,\psi):=(TS,\chi^S\psi),
\qquad
\chi^S(v):=\chi(Sv).
\]
The normal factor is written on the right:
\[
(T,\chi)=(T,1)(\id,\chi).
\]
This is the same right-coordinate convention used in the product of \(C(A)\).
Through the symplectic identification
\(\vartheta:V_A\xrightarrow{\sim}\widehat{V_A}\), this is the dual form of the
usual affine symplectic group with translations written by elements of \(V_A\).
\section{The tensor isomorphism and transported symmetry}
\label{sec:symmetric-tensor-equivalence}

For a finite group \(G\), write \(\Rep(G)\) for the category of
finite-dimensional complex representations of \(G\). Two finite groups are
\emph{isocategorical} if their representation categories are equivalent as
\(\C\)-linear tensor categories \cite{Davydov2001,EtingofGelaki}. We construct
an explicit isomorphism of tensor categories
\[
\Rep(\ASp(A))\cong \Rep(C(A)),
\]
and then identify the symmetry on \(\Rep(\ASp(A))\) transported to the ordinary
symmetry on \(\Rep(C(A))\).

\subsection{Representations of the affine symplectic group}
\label{subsec:affine-representations}

Representations of \(\ASp(A)\) are written using the weight decomposition for
its normal abelian subgroup \(\widehat{V_A}\). By Pontryagin duality these
weights are indexed by elements of \(V_A\), and \(\Sp(V_A)\) transports the
indices by its natural action.

Let \(\pi\) be a finite-dimensional representation of
\(\ASp(A)\) on \(U\). Restricting to the normal abelian subgroup
\[
\{(\id,\chi)\mid \chi\in \widehat{V_A}\}
\cong \widehat{V_A},
\]
gives the weight decomposition
\begin{equation}\label{eq:affine-weight-decomposition}
U=\bigoplus_{v\in V_A}U_v,
\qquad
U_v:=\{\xi\in U:\pi(\id,\chi)\xi=\chi(v)\xi
\text{ for all }\chi\in \widehat{V_A}\}.
\end{equation}
Here the label \(v\) denotes the character \(\iota_v\) of
\(\widehat{V_A}\).

Write
\[
\rho(T):=\pi(T,1).
\]
The relation
\[
(\id,\psi)(T,\chi)=(T,\chi)(\id,\psi^T),
\qquad
\psi^T(v):=\psi(Tv),
\]
implies, for \(\xi\in U_v\),
\[
\pi(\id,\psi)\rho(T)\xi
=
\rho(T)\pi(\id,\psi^T)\xi
=
\psi(Tv)\rho(T)\xi.
\]
Thus \(\rho(T)(U_v)\subseteq U_{Tv}\). Hence a representation of
\(\ASp(A)\) is the same as a \(V_A\)-graded vector space
\[
U=\bigoplus_{v\in V_A}U_v
\]
equipped with a representation \(\rho\) of \(\Sp(V_A)\) such that
\[
\rho(T)U_v\subseteq U_{Tv}
\qquad (T\in\Sp(V_A),\ v\in V_A).
\]

In these coordinates the element \((T,\chi)\in \ASp(A)\) acts on
\(U_v\) by
\begin{equation}\label{eq:split-action-graded}
(T,\chi)\cdot \xi=\chi(v)\,\rho(T)\xi,
\qquad \xi\in U_v.
\end{equation}
Conversely, any \(V_A\)-graded vector space with such a compatible
\(\Sp(V_A)\)-action becomes a representation through
\eqref{eq:split-action-graded}. Indeed, for \(\xi\in U_v\),
\[
(T,\chi)\cdot\bigl((S,\psi)\cdot \xi\bigr)
=
\psi(v)\chi(Sv)\rho(TS)\xi
=
(\chi^S\psi)(v)\rho(TS)\xi,
\]
which is exactly the action of \((TS,\chi^S\psi)\).

The tensor product carries the sum grading. If
\[
U=\bigoplus_{v\in V_A}U_v,
\qquad
W=\bigoplus_{w\in V_A}W_w,
\]
then
\begin{equation}\label{eq:affine-tensor-grading}
(U\otimes W)_x=\bigoplus_{v+w=x}U_v\otimes W_w.
\end{equation}
The ordinary symmetry in \(\Rep(\ASp(A))\) is the usual flip,
\begin{equation}\label{eq:affine-usual-symmetry}
\tau_{U,W}(\xi\otimes\eta)=\eta\otimes\xi.
\end{equation}

\subsection{The tensor isomorphism}\label{subsec:F-tensor}

With the graded description fixed, we construct the tensor isomorphism. The
functor keeps the same vector space and grading, but interprets the action
through the pair description of \(C(A)\). Its tensor constraint is defined
using the bicharacter \(\beta_A\).

Let \((U,\rho)\in \Rep(\ASp(A))\). On the same vector space
\(U\), define a \(C(A)\)-action by
\begin{equation}\label{eq:twisted-action}
(T,\lambda)\cdot \xi=\lambda(v)\,\rho(T)\xi,
\qquad \xi\in U_v.
\end{equation}
This gives an object of \(\Rep(C(A))\), denoted \(\cF(U)\). On morphisms
\(\cF\) is the identity.
We give \(\cF\) the tensor structure
\begin{equation}\label{eq:J-def}
J_{U,W}:\cF(U)\otimes\cF(W)\longrightarrow\cF(U\otimes W),
\qquad
J_{U,W}(\xi\otimes\eta):=\beta_A(v,w)\,(\xi\otimes\eta),
\end{equation}
for \(\xi\in U_v\) and \(\eta\in W_w\). The unit map is the identity.

Conversely, let \(X\in \Rep(C(A))\). Restricting the action to
\[
D_C=\{(\id,\chi)\in C(A)\mid \chi\in\widehat{V_A}\}
\]
gives the corresponding weight decomposition
\[
X=\bigoplus_{v\in V_A}X_v,
\qquad
X_v:=\{\xi\in X:(\id,\chi)\cdot \xi=\chi(v)\xi
\text{ for all }\chi\in \widehat{V_A}\}.
\]
The relation
\[
(\id,\chi)(T,\lambda)=(T,\lambda)(\id,\chi^T),
\qquad
\chi^T(v):=\chi(Tv),
\]
shows that \((T,\lambda)\cdot X_v\subseteq X_{Tv}\). Define a representation
\(\rho_X\) of \(\Sp(V_A)\) by
\begin{equation}\label{eq:inverse-rho}
\rho_X(T)\big|_{X_v}
:=
\lambda(v)^{-1}(T,\lambda)\big|_{X_v},
\end{equation}
where \((T,\lambda)\in C(A)\) is any lift of \(T\). This gives an object
\(\cF^{-1}(X)\in \Rep(\ASp(A))\), and on morphisms
\(\cF^{-1}\) is the identity.

Give \(\cF^{-1}\) the tensor structure
\begin{equation}\label{eq:inverse-J-def}
K_{X,Y}:\cF^{-1}(X)\otimes\cF^{-1}(Y)
\longrightarrow\cF^{-1}(X\otimes Y),
\qquad
K_{X,Y}(\xi\otimes\eta)
=\beta_A(v,w)^{-1}\xi\otimes\eta,
\end{equation}
for homogeneous \(\xi\in X_v\) and \(\eta\in Y_w\).
Equivalently,
\[
K_{X,Y}
=\cF^{-1}\bigl(J_{\cF^{-1}(X),\cF^{-1}(Y)}\bigr)^{-1}.
\]

\begin{theorem}\label{thm:tensor-equivalence}
The functor
\[
\cF:\Rep(\ASp(A))\to \Rep(C(A))
\]
with tensor structure \(J\) from \eqref{eq:J-def} is an isomorphism of tensor
categories. Its inverse is \(\cF^{-1}\) with tensor structure \(K\) from
\eqref{eq:inverse-J-def}.
\end{theorem}

\begin{proof}
First, \eqref{eq:twisted-action} is compatible with multiplication in
\(C(A)\). Indeed, if \(\xi\in U_v\), then
\[
(T,\lambda)\cdot\bigl((S,\mu)\cdot \xi\bigr)
=
\lambda(Sv)\mu(v)\rho(TS)\xi
=
(\lambda^S\mu)(v)\rho(TS)\xi,
\]
which is the action of \((T,\lambda)(S,\mu)\).

For \(\cF^{-1}\), the operator \(\rho_X(T)\) is independent of the chosen lift
of \(T\). If \((T,\lambda')\) is another lift, then
\(\lambda^{-1}\lambda'\in\widehat{V_A}\) and
\[
(T,\lambda')=(T,\lambda)(\id,\lambda^{-1}\lambda').
\]
The last factor acts on \(X_v\) by the scalar
\(\lambda(v)^{-1}\lambda'(v)\). Hence
\[
\lambda'(v)^{-1}(T,\lambda')\big|_{X_v}
=
\lambda(v)^{-1}(T,\lambda)\big|_{X_v}.
\]
Moreover, if \((T,\lambda)\) and \((S,\mu)\) are lifts of \(T\) and \(S\), then
\[
(T,\lambda)(S,\mu)=(TS,\lambda^S\mu),
\qquad
(\lambda^S\mu)(v)=\lambda(Sv)\mu(v),
\]
and for \(\xi\in X_v\),
\[
\rho_X(T)\rho_X(S)\xi
=
\lambda(Sv)^{-1}\mu(v)^{-1}(T,\lambda)(S,\mu)\xi
=
(\lambda^S\mu)(v)^{-1}(TS,\lambda^S\mu)\xi
=
\rho_X(TS)\xi.
\]
Hence \(\rho_X\) is a representation of \(\Sp(V_A)\). A morphism in
\(\Rep(C(A))\) is automatically \(D_C\)-equivariant and commutes with the full
\(C(A)\)-action; hence it preserves the grading and intertwines the operators
\(\rho_X(T)\). Thus \(\cF^{-1}\) is functorial.

For \(X\in \Rep(C(A))\), \(\xi\in X_v\), and \((T,\lambda)\in C(A)\),
\[
\lambda(v)\rho_X(T)\xi=(T,\lambda)\xi,
\]
hence \(\cF\cF^{-1}=\id\). Conversely, if
\((U,\rho)\in \Rep(\ASp(A))\), \(\xi\in U_v\), and
\((T,\lambda)\in C(A)\) is any lift of \(T\), then
\[
\rho_{\cF^{-1}(\cF(U))}(T)\xi
=
\lambda(v)^{-1}(T,\lambda)\cdot \xi
=
\rho(T)\xi,
\]
hence \(\cF^{-1}\cF=\id\).

The maps \(J_{U,W}\) are invertible and natural because morphisms preserve the
homogeneous components. Moreover, \(J_{U,W}\) intertwines the \(C(A)\)-action on
\(\cF(U)\otimes\cF(W)\) with the \(C(A)\)-action on \(\cF(U\otimes W)\):
for \((T,\lambda)\in C(A)\), \(\xi\in U_v\), and \(\eta\in W_w\),
\[
(T,\lambda)\cdot J_{U,W}(\xi\otimes\eta)
=\beta_A(v,w)\,\lambda(v+w)\,\rho(T)\xi\otimes\rho(T)\eta,
\]
while
\[
J_{U,W}\bigl((T,\lambda)\cdot \xi\otimes (T,\lambda)\cdot \eta\bigr)
=\beta_A(Tv,Tw)\,\lambda(v)\lambda(w)\,\rho(T)\xi\otimes\rho(T)\eta.
\]
These agree because the defining condition \eqref{eq:clifford-pair-description}
for the pair \((T,\lambda)\in C(A)\) gives
\[
\beta_A(v,w)\,\lambda(v+w)=\beta_A(Tv,Tw)\,\lambda(v)\lambda(w).
\]
The pentagon axiom for \(J\) holds because \(\beta_A\) is a \(2\)-cocycle on
\(V_A\): for homogeneous \(\xi\in U_v\), \(\eta\in W_w\), and
\(\zeta\in Z_z\),
\[
J_{U\otimes W,Z}\circ(J_{U,W}\otimes\id)
\quad\text{and}\quad
J_{U,W\otimes Z}\circ(\id\otimes J_{W,Z})
\]
both multiply by
\[
\beta_A(v,w)\beta_A(v+w,z)=\beta_A(v,w+z)\beta_A(w,z).
\]
The unit constraint follows from \(\beta_A(0,v)=\beta_A(v,0)=1\) and
\(\lambda(0)=1\). Thus \(J\) is a tensor structure on \(\cF\).
The identity expressing \(K\) as the transported inverse of \(J\) shows that
\(K\) is a tensor structure on \(\cF^{-1}\).

The tensor structures of the two composite functors are identities. Indeed,
on homogeneous tensors,
\[
\cF^{-1}(J_{U,W})\circ K_{\cF(U),\cF(W)}=\id_{U\otimes W},
\qquad
\cF(K_{X,Y})\circ J_{\cF^{-1}(X),\cF^{-1}(Y)}=\id_{X\otimes Y},
\]
because in each composition the factors \(\beta_A(v,w)^{-1}\) and
\(\beta_A(v,w)\) cancel. Consequently, \(K\) is the inverse tensor structure
to \(J\), and \(\cF\) and \(\cF^{-1}\) are inverse tensor functors.
\end{proof}

\begin{remark}
\label{rem:identical-character-tables}
For every finite abelian group \(A\), the tensor isomorphism \(\cF\) implies
that \(C(A)\) and \(\ASp(A)\) have identical ordinary character tables, up to
permutations of rows and columns. The row bijection sends \(\chi_U\) to
\(\chi_{\cF(U)}\). This follows from the induced isomorphism of based
Grothendieck rings. Theorem~\ref{thm:common-character-table} gives explicit
common row and column parameters and an entry-by-entry character formula.
\end{remark}

\subsection{The transported symmetry}\label{subsec:twisted-symmetry}

After the tensor isomorphism is established, the remaining point is which
symmetry on \(\Rep(\ASp(A))\) is carried to the ordinary symmetry on
\(\Rep(C(A))\). The usual flip in \(\Rep(\ASp(A))\) is not the transported
symmetry in general. The transported symmetry is
\begin{equation}\label{eq:transported-symmetry-definition}
c^\omega_{U,W}
:=
J_{W,U}\circ\tau_{\cF(U),\cF(W)}\circ J_{U,W}^{-1}.
\end{equation}
For homogeneous \(\xi\in U_v\) and \(\eta\in W_w\), this is
\begin{equation}\label{eq:omega-braiding}
c^\omega_{U,W}(\xi\otimes\eta)=\omega_A(w,v)\,\eta\otimes\xi.
\end{equation}
Indeed, the scalar is
\[
\frac{\beta_A(w,v)}{\beta_A(v,w)}=\omega_A(w,v).
\]

We denote the resulting symmetric category by
\(\Rep^\omega(\ASp(A))\). Since \(\omega_A\) is an
\(\Sp(V_A)\)-invariant bicharacter, \eqref{eq:omega-braiding} is equivariant
for the \(\Sp(V_A)\)-action; equivariance for \(\widehat{V_A}\) follows because
both sides have weight \(v+w\). Naturality follows because morphisms preserve
the \(V_A\)-grading. The two hexagon identities are exactly the bicharacter
identities in the two variables, and
\[
c^\omega_{W,U}c^\omega_{U,W}=\id
\]
follows from \(\omega_A(w,v)\omega_A(v,w)=1\). Hence it defines a symmetry.

\begin{theorem}\label{thm:symmetric-equivalence}
With the symmetry \eqref{eq:omega-braiding} on \(\Rep(\ASp(A))\), the functor
\[
\cF:\Rep^\omega(\ASp(A))\to \Rep(C(A))
\]
is an isomorphism of symmetric tensor categories.
\end{theorem}

\begin{proof}
The category \(\Rep^\omega(\ASp(A))\) has the same underlying
tensor category as \(\Rep(\ASp(A))\). Hence
Theorem~\ref{thm:tensor-equivalence} gives the tensor isomorphism. The
definition \eqref{eq:transported-symmetry-definition} is exactly the
compatibility condition with the ordinary symmetry on \(\Rep(C(A))\).
\end{proof}

Thus the symmetric comparison is
\[
\Rep^\omega(\ASp(A))\cong\Rep(C(A)),
\]
where \(\Rep(C(A))\) has its ordinary symmetry and \(\Rep^\omega(\ASp(A))\)
has the symmetry \eqref{eq:omega-braiding}. With the ordinary flip on
\(\Rep(\ASp(A))\), the tensor isomorphism of
Theorem~\ref{thm:tensor-equivalence} is not symmetric in general.
\section{Irreducible representations and character theory}
\label{sec:clifford-reps}

Using Theorem~\ref{thm:tensor-equivalence}, we classify the irreducible
representations of \(C(A)\), construct extensions of the kernel characters to
their inertia subgroups, and give common row and column parameters
for the character tables of \(C(A)\) and \(\ASp(A)\).

\subsection{Irreducible representations}\label{subsec:clifford-table-rows}

In the semidirect product
\(\ASp(A)=\Sp(V_A)\ltimes\widehat{V_A}\), we use the normal factor
\(\widehat{V_A}\). Its character group is canonically \(V_A\), and the induced
action of \(\Sp(V_A)\) on these characters is the natural action on \(V_A\).

For \(u\in V_A\), set
\[
\Sp(V_A)_u:=\Stab_{\Sp(V_A)}(u),
\qquad
\ASp(A)_u:=\Sp(V_A)_u\ltimes\widehat{V_A}.
\]
If \(\sigma\in\Irr(\Sp(V_A)_u)\), inflate \(\sigma\) to \(\ASp(A)_u\) and twist
it by the character \(u\) of \(\widehat{V_A}\):
\[
\widetilde\sigma_u(T,\chi):=\chi(u)\sigma(T),
\qquad
(T,\chi)\in \ASp(A)_u.
\]
Define
\[
\Pi(u,\sigma):=
\Ind_{\ASp(A)_u}^{\ASp(A)}(\widetilde\sigma_u).
\]

Mackey's little-group theorem \cite{Mackey1958} says that every irreducible representation of
\(\ASp(A)\) is isomorphic to some \(\Pi(u,\sigma)\), and that
\[
\Pi(u,\sigma)\simeq \Pi(u',\sigma')
\]
if and only if there is \(T\in\Sp(V_A)\) such that \(Tu=u'\) and
\(\sigma'\simeq {}^T\sigma\), where
\[
({}^T\sigma)(R):=\sigma(T^{-1}RT),
\qquad R\in\Sp(V_A)_{u'}.
\]
This is the equivalence relation for little-group parameters
\cite{Mackey1958,SerreRep}.
We use the term \emph{affine parameter} for such a pair \((u,\sigma)\),
always modulo this equivalence.

For an affine parameter \((u,\sigma)\), set
\[
\widetilde\Pi(u,\sigma):=\cF(\Pi(u,\sigma))\in\Rep(C(A)).
\]

\begin{theorem}
\label{thm:clifford-classification}
Every irreducible representation of \(C(A)\) is isomorphic to
\(\widetilde\Pi(u,\sigma)\) for a unique affine parameter \((u,\sigma)\). In
particular,
\[
\dim \widetilde\Pi(u,\sigma)
=
[\Sp(V_A):\Sp(V_A)_u]\dim\sigma.
\]
\end{theorem}

\begin{proof}
By Theorem~\ref{thm:tensor-equivalence}, \(\cF\) is an isomorphism of the
underlying abelian categories. Therefore it preserves and reflects
irreducibility and isomorphism classes. Combining this with the little-group
parametrization for \(\ASp(A)\) gives the result.
\end{proof}

The representations \(\widetilde\Pi(u,\sigma)\) also admit a direct induced
realization in \(C(A)\).

For \(u\in V_A\), let \(\chi_u\in\widehat{D_C}\) be given by
\[
\chi_u(\id,\mu):=\mu(u)\qquad ((\id,\mu)\in D_C).
\]
Conjugation by \((T,\lambda)\in C(A)\) sends \(\chi_u\) to the character
indexed by \(Tu\). Hence the inertia subgroup of \(\chi_u\) is
\[
I_u:=\{(T,\lambda)\in C(A):Tu=u\}.
\]
For \(\sigma\in\Irr(\Sp(V_A)_u)\), let \(\bar\sigma\) denote its inflation
along \(I_u\to\Sp(V_A)_u\).

\begin{theorem}
\label{thm:inertia-character-extension}
For every \(u\in V_A\), the character \(\chi_u\) extends canonically to
\(I_u\) by
\[
\widehat\chi_u(T,\lambda)=\lambda(u).
\]
Every associated factor set on \(\Sp(V_A)_u\) is therefore a coboundary.
Moreover, for every \(\sigma\in\Irr(\Sp(V_A)_u)\),
\begin{equation}\label{eq:intrinsic-induced-clifford-row}
\widetilde\Pi(u,\sigma)
\simeq
\Ind_{I_u}^{C(A)}(\widehat\chi_u\otimes\bar\sigma).
\end{equation}
\end{theorem}

\begin{proof}
If \((T,\lambda),(S,\mu)\in I_u\), then
\[
(\lambda^S\mu)(u)=\lambda(Su)\mu(u)=\lambda(u)\mu(u).
\]
Thus \(\widehat\chi_u\) is a character, and its restriction to \(D_C\) is
\(\chi_u\). This definition is independent of any section: it is the scalar by
which an element of \(I_u\) acts by conjugation on the Weyl line \(\C W_u\).

To describe the factor set, choose a normalized set-theoretic section
\(r:\Sp(V_A)_u\to I_u\) and write
\[
r(S)r(T)=c_u(S,T)r(ST),
\qquad c_u(S,T)\in D_C.
\]
Set
\[
\alpha_u(S,T):=\chi_u(c_u(S,T)),
\qquad
b_u(T):=\widehat\chi_u(r(T)).
\]
Applying \(\widehat\chi_u\) to
\(r(S)r(T)=c_u(S,T)r(ST)\) gives
\[
b_u(S)b_u(T)=\alpha_u(S,T)b_u(ST),
\]
which proves that \(\alpha_u\) is a coboundary.

Let \(E\) be the representation space of \(\sigma\), and use the standard
orbit realization
\[
\widetilde\Pi(u,\sigma)
=\bigoplus_{x\in\Sp(V_A)\cdot u}E_x,
\qquad E_x\simeq E,
\]
normalized at the base point \(u\). Its \(\chi_u\)-weight space is \(E_u\).
In this normalization, if \(g=(T,\lambda)\in I_u\), then \(Tu=u\), and the
action on this weight space is
\[
g\cdot e=\lambda(u)\sigma(T)e,
\qquad e\in E_u.
\]
Thus, as an \(I_u\)-module,
\[
E_u\simeq\widehat\chi_u\otimes\bar\sigma.
\]
There is therefore a natural \(C(A)\)-equivariant map
\[
\Ind_{I_u}^{C(A)}(E_u)
\longrightarrow
\widetilde\Pi(u,\sigma),
\qquad
g\otimes e\longmapsto g\cdot e.
\]
It is surjective because the \(C(A)\)-translates of \(E_u\) are precisely the
weight spaces \(E_x\). Moreover,
\[
[C(A):I_u]\dim E_u
=
[\Sp(V_A):\Sp(V_A)_u]\dim\sigma
=
\dim\widetilde\Pi(u,\sigma).
\]
Hence the map is an isomorphism, proving
\eqref{eq:intrinsic-induced-clifford-row}.
\end{proof}

\subsection{Conjugacy-class parameters}
\label{subsec:clifford-table-columns}

For \(T\in\Sp(V_A)\), let
\[
F_T:=\ker(T-\id)\leq V_A.
\]
Only the values of \(\lambda_g\) on \(F_T\) enter the conjugacy-class
parameters and the character formula. Therefore, for
\(g=(T,\lambda_g)\in C(A)\), set
\[
\ell_g:=\lambda_g|_{F_T}.
\]
Although \(\lambda_g\) need not be a character on \(V_A\),
Theorem~\ref{thm:intrinsic-clifford-conjugacy} shows that \(\ell_g\) is a
character of \(F_T\).

For \(\chi\in\widehat{V_A}\), recall the notation
\(\chi^T(x)=\chi(Tx)\).
For a subgroup \(B\leq V_A\), write
\[
\operatorname{Ann}(B)
:=\{\chi\in\widehat{V_A}:\chi|_B=1\}.
\]

Let
\[
\mathcal X_A
:=
\{(T,\ell):T\in\Sp(V_A),\ \ell\in\widehat{F_T}\}.
\]
The group \(\Sp(V_A)\) acts on \(\mathcal X_A\) by
\[
S\cdot(T,\ell)
:=
\bigl(STS^{-1},\ell\circ S^{-1}\bigr).
\]

\begin{theorem}\label{thm:intrinsic-clifford-conjugacy}
For every \(g=(T,\lambda_g)\in C(A)\), the restriction
\(\ell_g\) belongs to \(\widehat{F_T}\). Moreover, the assignment
\[
(T,\lambda_g)\longmapsto(T,\ell_g)
\]
induces a canonical bijection
\[
C(A)/\!\sim_{\mathrm{conj}}
\xrightarrow{\ \sim\ }
\Sp(V_A)\backslash\mathcal X_A.
\]
\end{theorem}

\begin{proof}
We first verify that the assignment is defined and compatible with
conjugation. If \(x,y\in F_T\), the defining relation
\eqref{eq:clifford-pair-description} gives
\[
\frac{\lambda_g(x+y)}{\lambda_g(x)\lambda_g(y)}
=
\frac{\beta_A(Tx,Ty)}{\beta_A(x,y)}
=1.
\]
Thus \(\ell_g\in\widehat{F_T}\). If \(h=(S,\mu)\in C(A)\) and
\(hgh^{-1}=(STS^{-1},\lambda')\), the product in the pair description gives
\[
\lambda'(x)
=
\mu(TS^{-1}x)\lambda_g(S^{-1}x)\mu(S^{-1}x)^{-1}.
\]
For \(x\in F_{STS^{-1}}\), one has \(TS^{-1}x=S^{-1}x\), and hence
\begin{equation}\label{eq:fixed-point-character-conjugation}
\ell_{hgh^{-1}}(x)=\ell_g(S^{-1}x).
\end{equation}

We claim that
\begin{equation}\label{eq:fixed-point-annihilator}
\operatorname{Ann}(F_T)
=
\{\chi^T\chi^{-1}:\chi\in\widehat{V_A}\}.
\end{equation}
The inclusion from right to left is immediate. The kernel of
\[
\widehat{V_A}\longrightarrow\widehat{V_A},
\qquad
\chi\longmapsto\chi^T\chi^{-1},
\]
is \(\operatorname{Ann}((T-\id)V_A)\), which has order
\[
|V_A/(T-\id)V_A|=|F_T|.
\]
The image therefore has order \(|V_A|/|F_T|\), equal to the order of
\(\operatorname{Ann}(F_T)\), proving
\eqref{eq:fixed-point-annihilator}. Conjugation by
\((\id,\chi)\in D_C\) sends \((T,\lambda)\) to
\[
\bigl(T,\chi^T\chi^{-1}\lambda\bigr).
\]
Since two functions over \(T\) differ by a character of \(V_A\), two
elements over the same \(T\) are conjugate by \(D_C\) exactly when their
restrictions to \(F_T\) agree.

We can now prove surjectivity. Fix \(T\in\Sp(V_A)\) and
\(\ell\in\widehat{F_T}\), and choose any
\((T,\lambda_0)\in C(A)\). The character
\[
\ell\,(\lambda_0|_{F_T})^{-1}\in\widehat{F_T}
\]
extends to a character \(\chi\in\widehat{V_A}\), by exactness of Pontryagin
duality for finite abelian groups. Then \((T,\chi\lambda_0)\in C(A)\) has
fixed-point character \(\ell\).

For injectivity, first suppose that two elements lie over the same \(T\) and
have the same fixed-point character. Equation~
\eqref{eq:fixed-point-annihilator} shows that they are conjugate by \(D_C\).
In general, first conjugate the symplectic part by a lift of an appropriate
element of \(\Sp(V_A)\), use \eqref{eq:fixed-point-character-conjugation}, and then
apply the fixed-\(T\) case. This proves the bijection with
\(\Sp(V_A)\backslash\mathcal X_A\).
\end{proof}

After choosing one representative \(T\) from each conjugacy class of
\(\Sp(V_A)\), the canonical parametrization in
Theorem~\ref{thm:intrinsic-clifford-conjugacy} becomes
\begin{equation}\label{eq:intrinsic-column-parametrization}
C(A)/\!\sim_{\mathrm{conj}}
\xrightarrow{\ \sim\ }
\bigsqcup_{[T]\subseteq\Sp(V_A)}
C_{\Sp(V_A)}(T)\backslash\widehat{F_T}.
\end{equation}

\begin{remark}
\label{rem:clifford-class-size}
Let \(g=(T,\lambda_g)\in C(A)\). Then
\[
|C_{C(A)}(g)|
=
|F_T|\,
|\Stab_{C_{\Sp(V_A)}(T)}(\ell_g)|.
\]
Consequently,
\[
|[g]_{C(A)}|
=
\frac{|C(A)|}
{|F_T|\,|\Stab_{C_{\Sp(V_A)}(T)}(\ell_g)|}.
\]
To see this, note that the image in \(\Sp(V_A)\) of an element centralizing
\(g\) belongs to \(C_{\Sp(V_A)}(T)\).
For \(S\in C_{\Sp(V_A)}(T)\), the conjugation calculation and
\eqref{eq:fixed-point-annihilator} in the proof of
Theorem~\ref{thm:intrinsic-clifford-conjugacy} show that a lift of \(S\) can
be modified by an element of \(D_C\) to centralize \(g\) if and only if
\(S\cdot\ell_g=\ell_g\). If one such lift exists, all centralizing lifts over
\(S\) form a coset of
\[
\{(\id,\chi)\in D_C:\chi^T=\chi\}.
\]
This group has order \(|F_T|\), by the same kernel-counting argument used in
the proof of Theorem~\ref{thm:intrinsic-clifford-conjugacy}. The centralizer
and class-size formulas follow.
\end{remark}

\subsection{The common character table}\label{subsec:clifford-characters}

Let \((u,\sigma)\) be an affine parameter. Mackey's restriction formula
\cite[\S7, Proposition~22]{SerreRep} gives the \(\widehat{V_A}\)-weight
decomposition
\[
M(u,\sigma)\simeq \bigoplus_{x\in \Sp(V_A)\cdot u}E_x,
\]
where each \(E_x\) is a copy of the representation space of \(\sigma\), and
\(\widehat{V_A}\) acts on \(E_x\) by the character \(\chi\mapsto\chi(x)\).
For each \(x\in \Sp(V_A)\cdot u\), choose \(T_x\in\Sp(V_A)\) with
\(T_xu=x\).

\begin{theorem}\label{thm:common-character-table}
The irreducible characters of \(C(A)\) and \(\ASp(A)\) have the common row
parameters
\[
\bigsqcup_{[u]\in \Sp(V_A)\backslash V_A}\Irr(\Sp(V_A)_u),
\]
and their conjugacy classes have the common column parameters
\[
\bigsqcup_{[T]\subseteq\Sp(V_A)}
C_{\Sp(V_A)}(T)\backslash\widehat{F_T}.
\]
Under these parameters, the two character tables agree entry by entry. Their
common value in row \((u,\sigma)\) and column \((T,\ell)\) is
\begin{equation}\label{eq:clifford-character}
\chi^{C(A)}_{(u,\sigma)}(T,\ell)
=
\chi^{\ASp(A)}_{(u,\sigma)}(T,\ell)
=
\sum_{x\in \Sp(V_A)\cdot u\cap F_T}
\ell(x)\,
\chi_\sigma(T_x^{-1}TT_x).
\end{equation}
\end{theorem}

\begin{proof}
The common row parameters are given by the little-group classification and
Theorem~\ref{thm:clifford-classification}. The column parameters for \(C(A)\)
are those of Theorem~\ref{thm:intrinsic-clifford-conjugacy}.

For \(\ASp(A)\), associate to \((T,\chi)\) the restriction
\(\chi|_{F_T}\). Two characters of \(V_A\) with the same restriction to
\(F_T\) differ by an element of \(\operatorname{Ann}(F_T)\), which equals the
image in \eqref{eq:fixed-point-annihilator}; hence the corresponding elements
of \(\ASp(A)\) are conjugate by its normal subgroup \(\widehat{V_A}\).
The conjugation calculation in the proof of
Theorem~\ref{thm:intrinsic-clifford-conjugacy} applies verbatim to the split
pair group and gives the natural centralizer action. Thus \(\ASp(A)\) has the
same column parameters.

In the induced realization, an element with symplectic part \(T\) carries
\(E_x\) to \(E_{Tx}\). Only the summands with \(x\in F_T\) contribute to the
trace. On such a summand, an element \((T,\chi)\in\ASp(A)\) has trace
\[
\chi(x)\chi_\sigma(T_x^{-1}TT_x)
=
\ell(x)\chi_\sigma(T_x^{-1}TT_x).
\]
For \(g=(T,\lambda_g)\in C(A)\), the functor \(\cF\) keeps the same graded
space and the same linear action of \(T\), while multiplying \(E_x\) by
\(\lambda_g(x)\). On the fixed-point subgroup,
\(\lambda_g(x)=\ell_g(x)\). Summing the diagonal blocks gives
\eqref{eq:clifford-character} for both groups.
\end{proof}

\begin{remark}
The common column parametrization gives a direct bijection between the
conjugacy classes of the two groups. If \(g=(T,\lambda_g)\in C(A)\), choose any
character \(\chi\in\widehat{V_A}\) such that
\[
\chi|_{F_T}=\lambda_g|_{F_T}.
\]
Then
\[
[(T,\lambda_g)]_{C(A)}
\longmapsto
[(T,\chi)]_{\ASp(A)}
\]
is independent of the extension \(\chi\) and of the representative of the
Clifford conjugacy class. Under the row correspondence induced by \(\cF\), it
preserves every character value.
\end{remark}

The character formula \eqref{eq:clifford-character} is independent of the
choices of the elements \(T_x\). Replacing
\(T_x\) by \(T_xh\), with \(h\in\Sp(V_A)_u\), conjugates
\(T_x^{-1}TT_x\) inside \(\Sp(V_A)_u\), and hence does not change the value of
\(\chi_\sigma\). It is also constant on the centralizer orbit of \(\ell\), as
required by the column parametrization.

\subsection{Cyclic \texorpdfstring{$2$}{2}-power groups}
\label{subsec:cyclic-2power-example}

We first spell out the formulas for the cyclic \(2\)-power family. This is the
basic example in which the orbit structure is richer than in the elementary
abelian \(2\)-group case but still controlled by elementary linear algebra over
\(\Z/2^k\Z\).

Set
\[
R_k:=\Z/2^k\Z,\qquad
A:=R_k,\qquad
\zeta_k:=e^{2\pi i/2^k}.
\]
Throughout this subsection \(k>1\), and we write
\[
C_k:=C(A).
\]
We identify \(\widehat{A}\) with \(R_k\) by
\[
\xi\longmapsto \chi_\xi,
\qquad
\chi_\xi(x)=\zeta_k^{\,\xi x}.
\]
Thus \(V_A\simeq R_k^2\), and
\begin{equation}\label{eq:omega-cyclic-2power-section}
\omega_A\bigl((x,\xi),(y,\eta)\bigr)=\zeta_k^{\,\xi y-\eta x}.
\end{equation}
In these coordinates
\[
\Sp(V_A)\simeq \SL_2(R_k),
\qquad
|\Sp(V_A)|=3\cdot 2^{3k-2}.
\]
Indeed, if \(M\in M_2(R_k)\), then
\[
\omega_A(Mv,Mw)=\omega_A(v,w)^{\det(M)},
\]
hence preservation of \(\omega_A\) is exactly the condition \(\det(M)=1\).

The \(\SL_2(R_k)\)-orbits in \(V_A\) are controlled by the \(2\)-adic valuation.
For \(x\in R_k\), put
\[
\nu_2(x):=\max\{r:x\in 2^rR_k\},
\qquad
\nu_2(0):=k,
\]
and, for \(v=(x,\xi)\), set
\[
\nu(v):=\min\{\nu_2(x),\nu_2(\xi)\}.
\]
The orbits are
\[
\{0\},
\qquad
\mathcal O_r:=\{v\in R_k^2:\nu(v)=r\},
\qquad
0\leq r\leq k-1.
\]
Let \(R_n=\Z/2^n\Z\). A primitive vector \(\bar w\in R_{k-r}^2\), lifted to
some \(w\in R_k^2\), gives an element \(2^rw\in\mathcal O_r\), and this is
independent of the lift. This gives a bijection from the primitive vectors of
\(R_{k-r}^2\) onto \(\mathcal O_r\): indeed, \(2^rw=2^rw'\) in \(R_k^2\) if
and only if \(w\equiv w'\pmod{2^{k-r}}\), and primitivity is the condition that
at least one coordinate be a unit modulo \(2\). The transitivity follows from
the primitive case over \(R_k\). If \(w=(a,b)\in R_k^2\) is primitive, then
\(aR_k+bR_k=R_k\), so one can choose \(c,d\in R_k\) with \(ad-bc=1\), and the
matrix with first column \(w\) carries \((1,0)\) to \(w\). Applying this to a
primitive lift of \(\bar w\) carries \(u_r=(2^r,0)\) to \(2^rw\).
We use the representatives
\[
u_r:=(2^r,0).
\]
The size of the \(r\)-th orbit is
\begin{equation}\label{eq:cyclic-orbit-size}
|\mathcal O_r|
=2^{2(k-r)}-2^{2(k-r-1)}
=3\cdot2^{2(k-r)-2}.
\end{equation}
Its stabilizer in \(\SL_2(R_k)\) is
\begin{equation}\label{eq:cyclic-stabilizer}
G_r=
\left\{
\begin{pmatrix}a&b\\ c&d\end{pmatrix}\in \SL_2(R_k):
a\equiv1 \pmod{2^{k-r}},\quad
c\equiv0 \pmod{2^{k-r}}
\right\},
\qquad
|G_r|=2^{k+2r}.
\end{equation}
In particular,
\[
G_0=
\left\{
\begin{pmatrix}1&b\\0&1\end{pmatrix}:b\in R_k
\right\}
\simeq (R_k,+).
\]

The row parametrization of Theorem~\ref{thm:clifford-classification} now has
one zero-orbit block and \(k\) valuation blocks. The zero-orbit rows are
\[
\widetilde\Pi_{\mathrm{zero}}(\sigma):=\widetilde\Pi(0,\sigma),
\qquad
\sigma\in\Irr(\SL_2(R_k)),
\]
and, for every \(0\leq r\leq k-1\), the \(r\)-th valuation block consists of
\[
\widetilde\Pi_r(\tau):=\widetilde\Pi(u_r,\tau),
\qquad
\tau\in\Irr(G_r).
\]
We write their characters as
\[
\widetilde\chi_{\mathrm{zero},\sigma}:=\widetilde\chi_{(0,\sigma)},
\qquad
\widetilde\chi_{r,\tau}:=\widetilde\chi_{(u_r,\tau)}.
\]
Their dimensions are
\begin{equation}\label{eq:cyclic-row-dimensions}
\dim\widetilde\Pi_{\mathrm{zero}}(\sigma)=\dim\sigma,
\qquad
\dim\widetilde\Pi_r(\tau)
=|\mathcal O_r|\dim\tau
=3\cdot2^{2(k-r)-2}\dim\tau.
\end{equation}
The stabilizer character theory of the primitive block \(r=0\) is completely
explicit. The irreducible characters of \(G_0\) are indexed by
\(\ell\in R_k\) and given by
\[
\tau_\ell
\left(
\begin{pmatrix}1&b\\0&1\end{pmatrix}
\right)
:=
\zeta_k^{\,\ell b}.
\]
For \(r>0\), the ordinary character theory of \(G_r\) is a separate finite
group problem. The present description uses the valuation orbit
\(\mathcal O_r\), its stabilizer \(G_r\), and the character formula
\eqref{eq:cyclic-character-section}.

The character formula specializes as follows in each valuation block. Write
\(g\in C_k\) as \(g=(T,\lambda_g)\) in the pair description, with
\(T\in \SL_2(R_k)\), and put \(\ell_g=\lambda_g|_{F_T}\). For
\(x\in\mathcal O_r\), choose
\(T_x\in \SL_2(R_k)\) with
\(T_xu_r=x\). Then Theorem~\ref{thm:common-character-table} gives
\begin{equation}\label{eq:cyclic-character-section}
\widetilde\chi_{r,\tau}(g)
=
\sum_{x\in\mathcal O_r\cap F_T}
\ell_g(x)\,
\chi_\tau(T_x^{-1}TT_x).
\end{equation}
The fixed points in \(\mathcal O_r\) can also be read after reducing modulo
\(2^{k-r}\): they correspond to primitive vectors in
\[
\ker(\overline T_r-I)\subseteq (\Z/2^{k-r}\Z)^2,
\]
where \(\overline T_r\) is the reduction of \(T\) modulo \(2^{k-r}\). If this
kernel contains no primitive vector, then the \(r\)-th valuation block has zero
character value at \(g\).

The kernel values are also explicit. For \(v\in V_A\),
\[
\widetilde\chi_{\mathrm{zero},\sigma}(K_v)=\dim\sigma,
\]
while
\[
\widetilde\chi_{r,\tau}(K_v)
=\dim\tau\sum_{x\in\mathcal O_r}\omega_A(v,x).
\]
Writing \(L_j:=(2^jR_k)^2\), the orbit is \(L_r\setminus L_{r+1}\). The sum of
the character \(x\mapsto\omega_A(v,x)\) over \(L_j\) is \(|L_j|\) when
\(\nu(v)\geq k-j\), and is zero otherwise. Subtracting the sums over \(L_r\)
and \(L_{r+1}\) gives
\begin{equation}\label{eq:cyclic-kernel-values}
\widetilde\chi_{r,\tau}(K_v)
=
\begin{cases}
|\mathcal O_r|\dim\tau,& \nu(v)\geq k-r,\\
-2^{2(k-r)-2}\dim\tau,& \nu(v)=k-r-1,\\
0,& \nu(v)<k-r-1.
\end{cases}
\end{equation}
For \(r=0\), this says that the primitive block is nonzero on the kernel only
at \(v=0\) and on the elements of valuation \(k-1\).

The conjugacy classes lying over the conjugacy class of a fixed
\(T\in\SL_2(R_k)\) are the orbits
\[
C_{\SL_2(R_k)}(T)\backslash\widehat{F_T}.
\]
If \(\det(T-I)\) is odd, then \(F_T=0\), and there is a single Clifford
conjugacy class lying over the base class of \(T\). For \(T=\id\), the symplectic
identification \(\widehat{V_A}\simeq V_A\) recovers exactly the \(k+1\) kernel
classes
\[
\{0\},\mathcal O_0,\ldots,\mathcal O_{k-1}.
\]
Thus the cyclic case has \(k+1\) row blocks and \(k+1\) kernel conjugacy
classes, both ordered by the same valuation filtration.

\subsection{Elementary abelian \texorpdfstring{$2$}{2}-groups}
\label{subsec:elementary-2group-example}

The other basic \(2\)-primary family is
\[
A=\F_2^m,\qquad m\geq1.
\]
Here \(V_A\) is a symplectic vector space over \(\F_2\), and we use the
notation
\[
V_m:=V_A,\qquad C_m:=C(A).
\]

Let \(A^*:=\Hom_{\F_2}(A,\F_2)\) be the linear dual. The assignment
\[
\xi\longmapsto \chi_\xi,\qquad \chi_\xi(x):=(-1)^{\xi(x)},
\]
identifies \(A^*\) with the Pontryagin dual \(\widehat{A}\); we use it to write
\[
V_m=A\oplus A^*.
\]
This is a \(2m\)-dimensional \(\F_2\)-vector space. Fix a basis
\(e_1,\dots,e_m\) of \(A\); together with the dual basis of \(A^*\), it
identifies \(V_m\) with \(\F_2^{2m}\).

The symplectic form is
\[
\langle\,\cdot\,,\,\cdot\,\rangle:V_m\times V_m\to\F_2,
\qquad
\bigl\langle (x,\xi),(y,\eta)\bigr\rangle:=\xi(y)+\eta(x).
\]
It is alternating and nondegenerate, and it is related to the global
bicharacter by
\[
\omega_A=(-1)^{\langle\,\cdot\,,\,\cdot\,\rangle}.
\]
We use \(\langle\,\cdot\,,\,\cdot\,\rangle\) for computations over \(\F_2\)
and \((-1)^{\langle\,\cdot\,,\,\cdot\,\rangle}\) for the corresponding
character values. The isometry group of
\(\langle\,\cdot\,,\,\cdot\,\rangle\) is \(\Sp(V_m)\), which in the fixed
basis is \(\Sp(2m,2)\). We also use
\(\langle\,\cdot\,,\,\cdot\,\rangle\) to identify \(V_m\) with its character
group:
\begin{equation}\label{eq:Vm-character-identification}
v\longmapsto \chi_v,\qquad
\chi_v(x):=(-1)^{\langle v,x\rangle}.
\end{equation}
Under this identification, the \(\Sp(V_m)\)-orbits in \(V_m\) are the same as
the \(\Sp(V_m)\)-orbits in \(\Irr(V_m)\).
The orbit structure here is simpler than in the cyclic \(2\)-power case:
\(\Sp(V_m)\) is transitive on \(V_m\setminus\{0\}\), hence it has exactly two
orbits on \(V_m\). This two-orbit decomposition is the input for the
Basheer--Moori discussion. By \cite[Theorem~6.1]{GalindoClifford2026}, the
projection
\[
C_m\to\Sp(V_m)
\]
splits for \(m=1\) and is nonsplit for \(m\geq2\).

The two orbits reduce the row parametrization to
\[
V_m=\{0\}\sqcup\bigl(V_m\setminus\{0\}\bigr).
\]
Set
\[
u_m=(e_1,0)\in A\oplus A^*.
\]
This represents the nonzero orbit, and we put
\[
P_m:=\Stab_{\Sp(V_m)}(u_m).
\]
Since the line \(\F_2u_m=\{0,u_m\}\) contains \(u_m\) as its only nonzero
vector, \(P_m\) is at the same time the stabilizer of this line; it is the
maximal parabolic subgroup
\begin{equation}\label{eq:Pm-parabolic-structure}
P_m\simeq 2^{2m-1}:\Sp(2m-2,2),
\end{equation}
where \(2^{2m-1}\) denotes an elementary abelian \(2\)-group. To see this,
choose a symplectic decomposition
\[
V_m=\F_2u_m\oplus W\oplus\F_2f,
\qquad
\langle u_m,f\rangle=1,
\]
with \(W\) orthogonal to \(\F_2u_m\oplus\F_2f\). The action of \(P_m\) on
\(u_m^\perp/\F_2u_m\simeq W\) gives a surjection onto \(\Sp(W)\). For
\(a\in W\) and \(c\in\F_2\), define
\[
\begin{aligned}
n_{a,c}(u_m)&=u_m,\\
n_{a,c}(w)&=w+\langle a,w\rangle u_m &&(w\in W),\\
n_{a,c}(f)&=f+a+cu_m.
\end{aligned}
\]
A direct calculation shows that these are precisely the elements acting
trivially on \(u_m^\perp/\F_2u_m\simeq W\), and
\[
n_{a,c}n_{b,d}
=n_{a+b,c+d+\langle a,b\rangle}.
\]
Because an alternating form is symmetric in characteristic \(2\), these
elements commute; moreover,
\[
n_{a,c}^2=n_{0,\langle a,a\rangle}=1.
\]
They therefore form an elementary abelian group of order
\(2^{\dim W+1}=2^{2m-1}\). Extending each element of \(\Sp(W)\) by fixing
\(u_m\) and \(f\) splits the quotient map, and every element of \(P_m\) has a
unique factorization of the form \(n_{a,c}S\), proving
\eqref{eq:Pm-parabolic-structure}.

The row parametrization of Theorem~\ref{thm:clifford-classification} now has
only two blocks, the zero-orbit and nonzero-orbit blocks. The zero-orbit rows
are
\[
\widetilde\Pi_{\mathrm{zero}}(\sigma):=\widetilde\Pi(0,\sigma),
\qquad
\sigma\in\Irr(\Sp(V_m)),
\]
and the nonzero-orbit rows are
\[
\widetilde\Pi_{\mathrm{nz}}(\tau):=\widetilde\Pi(u_m,\tau),
\qquad
\tau\in\Irr(P_m).
\]
We write the corresponding characters as
\[
\widetilde\chi_{\mathrm{zero},\sigma}:=\widetilde\chi_{(0,\sigma)},
\qquad
\widetilde\chi_{\mathrm{nz},\tau}:=\widetilde\chi_{(u_m,\tau)}.
\]
Their dimensions are
\begin{equation}\label{eq:elementary-2group-row-dimensions}
\dim\widetilde\Pi_{\mathrm{zero}}(\sigma)=\dim\sigma,
\qquad
\dim\widetilde\Pi_{\mathrm{nz}}(\tau)
=(2^{2m}-1)\dim\tau.
\end{equation}
Thus every nonzero-orbit irreducible degree is divisible by \(2^{2m}-1\). The
ordinary character theories of \(\Sp(V_m)\) and \(P_m\) are separate
finite-group problems; the present description uses them together with the two-orbit
decomposition and the character formula
\eqref{eq:elementary-2group-character}.

The same two-orbit description gives a compact block form for the character
table. The zero-orbit block is the ordinary character table of \(\Sp(V_m)\),
inflated along
\[
C_m\to \Sp(V_m).
\]
The nonzero-orbit block is controlled by the characters of \(P_m\) and by fixed
nonzero vectors of the symplectic part \(T\).

The character formula specializes as follows. Write \(g\in C_m\) as
\(g=(T,\lambda_g)\) in the pair description, with \(T\in\Sp(V_m)\), and put
\(\ell_g=\lambda_g|_{F_T}\). For each \(x\in V_m\setminus\{0\}\), choose
\(T_x\in\Sp(V_m)\) with \(T_xu_m=x\). Then
Theorem~\ref{thm:common-character-table} gives
\begin{equation}\label{eq:elementary-2group-character}
\widetilde\chi_{\mathrm{nz},\tau}(g)
=
\sum_{x\in F_T\setminus\{0\}}
\ell_g(x)\,
\chi_\tau(T_x^{-1}TT_x).
\end{equation}
The fixed-point set indexing the sum is
\[
F_T\setminus\{0\}.
\]
Consequently, if \(d(T):=\dim_{\F_2}F_T\), then the nonzero-orbit
character value is a sum over \(2^{d(T)}-1\) terms, and it vanishes whenever
\(d(T)=0\).

On the kernel \(V_m\subset C_m\), the values are completely explicit. For
\(v\in V_m\),
\[
\widetilde\chi_{\mathrm{zero},\sigma}(K_v)=\dim\sigma,
\]
whereas
\begin{equation}\label{eq:elementary-kernel-values}
\widetilde\chi_{\mathrm{nz},\tau}(K_v)
=\dim\tau\sum_{x\in V_m\setminus\{0\}}(-1)^{\langle v,x\rangle}
=
\begin{cases}
(2^{2m}-1)\dim\tau,& v=0,\\
-\dim\tau,& v\neq0.
\end{cases}
\end{equation}
Thus the restriction of a nonzero-orbit character to the kernel is a fixed
multiple of the sum of all nontrivial kernel characters. After division by
\(\dim\tau\), the two kernel values are independent of \(\tau\). These values
give the identity Fischer-matrix calculation in
Subsection~\ref{subsec:basheer-moori-conjecture}.

The conjugacy-class description is similarly compact. For a fixed
\(T\in\Sp(V_m)\), the group \(\widehat{F_T}\) is an \(\F_2\)-vector space of
dimension \(d(T)\), and the conjugacy classes of \(C_m\) lying over the conjugacy
class of \(T\) are
\[
C_{\Sp(V_m)}(T)\backslash\widehat{F_T}.
\]
In particular, if \(T\) has no nonzero fixed vector, then \(F_T=0\), and there
is a single Clifford conjugacy class lying over the base conjugacy class of \(T\).
For \(T=\id\), one gets the two kernel conjugacy classes, corresponding under
\(\widehat{V_m}\simeq V_m\) to \(0\) and \(V_m\setminus\{0\}\).

\subsection{The Basheer--Moori conjecture}
\label{subsec:basheer-moori-conjecture}

Keep the notation of Subsection~\ref{subsec:elementary-2group-example}, and
assume \(m\ge2\). Then \(C_m\) is the nonsplit extension
\begin{equation}\label{eq:elementary-clifford-extension}
1\to V_m\to C_m\to \Sp(V_m)\simeq\Sp(2m,2)\to 1.
\end{equation}
This is the family \(2^{2m}\!\cdot\Sp(2m,2)\) considered by Basheer and
Moori. To formulate their conjecture in our notation, we first recall the
relevant Clifford--Fischer terminology.

Clifford--Fischer theory is a character-table method for finite group
extensions, developed from Clifford theory and used especially when a group has
a large normal subgroup. In the terminology of
\cite{BasheerMooriSurvey2015}, the construction starts with an extension
\[
1\to N\to G\to \overline G\to 1.
\]
The quotient \(\overline G\) acts on \(\Irr(N)\), and the orbits organize the
irreducible characters of \(G\) into blocks. If
\(\theta\in\Irr(N)\) is an orbit representative, one forms its inertia group
\[
I_G(\theta)=\{g\in G:{}^g\theta=\theta\}
\]
and the inertia factor \(I_G(\theta)/N\). The corresponding block of the
character table is then assembled from two kinds of data. The first is the
character table, ordinary or projective, of the inertia factor. The second is
the family of Fischer matrices. Concretely, if
\(\overline{\mathcal K}\) is a conjugacy class of \(\overline G\), then its
preimage in \(G\) usually splits into several conjugacy classes. The Fischer
matrix attached to \(\overline{\mathcal K}\) has columns corresponding to the
classes of \(G\) lying over \(\overline{\mathcal K}\), and its rows are organized by
orbit representatives in \(\Irr(N)\) and their inertia factors. Its entries
record how the character tables of the inertia factors
contribute to the values of irreducible characters of \(G\) on those classes.
Thus the Fischer matrices provide the remaining character-value data once the
inertia-factor tables are known; see
\cite[Sections~1 and~3]{BasheerMooriSurvey2015} for the precise normalization.
The subtle point is that \(\theta\) need
not extend to an ordinary character of \(I_G(\theta)\). In that case the
Clifford--Fischer construction may require a projective character table of the
inertia factor, with a nontrivial factor set.

For the present family, Basheer and Moori write
\[
G_m=2^{2m}\!\cdot\Sp(2m,2),\qquad m\ge2,
\]
with elementary abelian kernel \(2^{2m}\) \cite{BasheerMooriNonSplit2015,
BasheerMooriSurvey2015}. In our notation this is the Clifford extension
\eqref{eq:elementary-clifford-extension}.
Applying this construction with \(G=C_m\) and \(N=V_m\), the character table of
\(C_m\) is decomposed into blocks indexed by the \(\Sp(V_m)\)-orbits in
\(\Irr(V_m)\). Through \eqref{eq:Vm-character-identification}, these correspond
to the two orbits in \(V_m\) from
Subsection~\ref{subsec:elementary-2group-example}:
\[
\{0\},\qquad V_m\setminus\{0\}.
\]
Hence there are two Clifford--Fischer blocks. The first is attached to the
trivial character \(\chi_0\) and has inertia factor
\[
H_1=I_{C_m}(\chi_0)/V_m\simeq \Sp(V_m)\simeq\Sp(2m,2).
\]
The nontrivial block is attached to the character \(\chi_{u_m}\), where
\(u_m\in V_m\setminus\{0\}\) is the representative chosen in
Subsection~\ref{subsec:elementary-2group-example}. Its inertia factor is
\[
H_2=I_{C_m}(\chi_{u_m})/V_m\simeq P_m.
\]
We denote this block by
\[
\mathcal B_{\mathrm{nz}}(C_m\mid V_m).
\]
Because it comes from a nontrivial character of \(V_m\), the
Clifford--Fischer method does not by itself decide whether the ordinary
character table of \(H_2\) is sufficient; a projective character table of
\(H_2\), with a nontrivial factor set, could be required. Basheer and Moori
conjectured that this projective table is unnecessary for every \(m\ge2\)
\cite[Section~5.2]{BasheerMooriSurvey2015}. In the same section they prove this
claim for \(m=2,3,4,5,6\), and report complete Clifford--Fischer computations
of the inertia factors, Fischer matrices, and character tables for
\[
G_3=2^6\!\cdot\Sp(6,2),
\qquad
G_4=2^8\!\cdot\Sp(8,2);
\]
see also \cite{BasheerMooriNonSplit2015}.

\begin{theorem}
\label{thm:basheer-moori-conjecture}
Let \(m\geq2\). The nontrivial Clifford--Fischer block of \(C_m\) relative to
\(V_m\) is
\[
\mathcal B_{\mathrm{nz}}(C_m\mid V_m)
=
\bigl\{\widetilde\chi_{\mathrm{nz},\tau}:\tau\in\Irr(P_m)\bigr\},
\]
indexed by the \emph{ordinary} irreducible characters of its inertia factor
\[
H_2=I_{C_m}(\chi_{u_m})/V_m\simeq P_m,
\]
with values in the fixed-point columns \((T,\ell)\) given by
\eqref{eq:elementary-2group-character}. No projective character table of
\(H_2\) is required.
\end{theorem}

\begin{proof}
Since \(V_m\) is abelian, it acts trivially by conjugation on \(\Irr(V_m)\), so
the action of \(C_m\) on \(\Irr(V_m)\) factors through
\(C_m/V_m\simeq\Sp(V_m)\). Under
\eqref{eq:Vm-character-identification}, this is the natural action of
\(\Sp(V_m)\) on \(V_m\). Hence
\[
H_2=I_{C_m}(\chi_{u_m})/V_m
=\Stab_{\Sp(V_m)}(\chi_{u_m})
=\Stab_{\Sp(V_m)}(u_m)
=P_m.
\]
The parabolic structure of \(P_m\) is \eqref{eq:Pm-parabolic-structure}.

By Theorem~\ref{thm:clifford-classification}, the nonzero-orbit irreducibles
are
\[
\{\widetilde\Pi(u_m,\tau):\tau\in\Irr(P_m)\}.
\]
Passing to characters gives the nonzero-orbit block, which is the nontrivial
Clifford--Fischer block, with values
\eqref{eq:elementary-2group-character}. The inducing character
\(\chi_{u_m}\) extends to its inertia subgroup by
Theorem~\ref{thm:inertia-character-extension}, and \(\cF\) transports the
ordinary induced representation \(\Pi(u_m,\tau)\) of \(\ASp(\F_2^m)\) to
\(\widetilde\Pi(u_m,\tau)\). Hence the block is indexed by the ordinary
characters \(\tau\in\Irr(P_m)\), and no projective character table of \(H_2\)
enters. The nonsplitting remains a property of the group extension, but it
creates neither projective little-group data nor affine ambiguity in the
fixed-point parametrization of the columns.
\end{proof}

The projective ambiguity in the Clifford--Fischer construction is not removed
by choosing a splitting of the Clifford extension, which does not exist for
\(m\ge2\). In the present approach it is removed by transporting ordinary
induction data from \(\ASp(\F_2^m)\) through the tensor isomorphism \(\cF\).

The identity coset already shows the same mechanism at the level of Fischer
matrices. The kernel values \eqref{eq:elementary-kernel-values} give
\[
\widetilde\chi_{\mathrm{nz},\tau}(1)
=(2^{2m}-1)\dim\tau,
\qquad
\widetilde\chi_{\mathrm{nz},\tau}(K_v)=-\dim\tau
\quad(v\neq0).
\]
After dividing by \(\dim\tau\), the nontrivial block contributes the row
\[
\bigl(2^{2m}-1,\,-1\bigr),
\]
which matches the nontrivial row of the identity Fischer matrix described by
Basheer and Moori for \(2^{2m}\!\cdot\Sp(2m,2)\), after the normalization by
\(\dim\tau\); see
\cite[Section~5.2]{BasheerMooriSurvey2015}. The parametrization in
Theorem~\ref{thm:clifford-classification} gives
the inertia factor \(H_2\) and recovers the basic Fischer-matrix data from the
Clifford character formula.

Using \eqref{eq:Vm-character-identification} to identify the translation
subgroup with \(V_m\), one has
\[
\ASp(\F_2^m)\cong\Sp(V_m)\ltimes V_m.
\]
Remark~\ref{rem:identical-character-tables} specializes here to Fischer's
character-table coincidence for \(C_m\) and \(\ASp(\F_2^m)\)
\cite{Fischer1988Identical}. In the present description, the two tables are
computed from the same induced row data and the same fixed-point column
parameters.

Mastel's independent Clifford-theoretic analysis overlaps with this
elementary-abelian specialization. In rank \(m+1\),
\cite[Lemma~IV.6]{Mastel2026Clifford} gives a central involution
\(z_{m+1}\in P_{m+1}\) and an isomorphism
\begin{equation}\label{eq:elementary-stabilizer-quotient}
P_{m+1}/\langle z_{m+1}\rangle\simeq \ASp(\F_2^m).
\end{equation}
Section~V of \cite{Mastel2026Clifford} then uses this quotient to construct a
lifting of irreducible characters from rank \(m\) to rank \(m+1\).
Let \(\chi\in\Irr(C_m)\), and let
\(\chi^{\mathrm{aff}}\in\Irr(\ASp(\F_2^m))\) be the corresponding
affine character under the character-table identification in
Remark~\ref{rem:identical-character-tables}. Inflating
\(\chi^{\mathrm{aff}}\) along \(P_{m+1}\to \ASp(\F_2^m)\) gives an
irreducible character \(\overline\chi\in\Irr(P_{m+1})\). The nonzero-orbit row in
rank \(m+1\)
\[
\widetilde\chi_{\mathrm{nz},\overline\chi}^{(m+1)}
\in\Irr(C_{m+1}),
\]
where the superscript records the rank, is the nonzero-orbit character attached
to Mastel's lifted parameter. Its irreducibility is the nonzero-orbit case of
Theorem~\ref{thm:clifford-classification}. Thus, combining Mastel's quotient
isomorphism with Theorems~\ref{thm:clifford-classification} and
\ref{thm:tensor-equivalence} recovers his lifting construction from rank
\(m\) to rank \(m+1\). This does not reprove Mastel's quotient isomorphism;
rather, it places the resulting lift inside the same tensor-categorical
parametrization that gives Theorem~\ref{thm:basheer-moori-conjecture}.
\section{The Weyl algebra and tensor-power commutants}
\label{sec:weil-algebra}

The algebraic counterpart of the Clifford action is the twisted group algebra
\[
\cA=\C_{\beta_A}[V_A].
\]
The Weyl operators realize \(\cA\) on \(\HH=\C[A]\), producing the projective
Weil action. We use the tensor isomorphism \(\cF\) to describe both its
tensor-power commutants and the adjoint-action commutants; trace duality relates
the latter to doubled tensor powers of the former.

\subsection{The twisted group algebra and its Clifford action}
\label{subsec:twisted-weyl-algebra}

We first fix the algebra and the strict Clifford action on it. The algebra
\(\cA\) has basis \(\{u_v\}_{v\in V_A}\) and multiplication
\begin{equation}\label{eq:twisted-multiplication}
u_vu_w=\beta_A(v,w)\,u_{v+w}.
\end{equation}
For \(g=(T,\lambda)\in C(A)\), define
\begin{equation}\label{eq:alpha-action}
\alpha_g(u_v):=\lambda(v)\,u_{Tv}.
\end{equation}
The defining relation of \(C(A)\) says exactly that \(\alpha_g\) is an algebra
automorphism of \(\cA\). Moreover, \(g\mapsto\alpha_g\) is a genuine group
action on \(\cA\):
\[
\alpha_{(T,\lambda)}\alpha_{(S,\mu)}(u_v)
=
\lambda(Sv)\mu(v)u_{TSv}
=
\alpha_{(TS,\lambda^S\mu)}(u_v).
\]
The multiplication gives the commutation rule
\[
u_vu_w=\omega_A(v,w)\,u_wu_v.
\]

Let \(\HH=\C[A]\), with basis \(\{\ket{x}\}_{x\in A}\). The action of \(\cA\)
on \(\HH\) is
\begin{equation}\label{eq:weyl-action-on-H}
u_{(a,\chi)}\cdot \ket{x}:=\chi(x)\,\ket{x+a}.
\end{equation}
Under this action, \(u_{(a,\chi)}\) is the translation-character operator
\(W_{(a,\chi)}\). They span \(\End_\C(\HH)\). Indeed, if
\[
E_{a+b,b}\ket{x}=
\begin{cases}
\ket{a+b},&x=b,\\
0,&x\neq b,
\end{cases}
\]
then character orthogonality gives
\[
E_{a+b,b}
=
\frac{1}{|A|}\sum_{\chi\in\widehat{A}}\chi(b)^{-1}W_{(a,\chi)}.
\]
Thus the representation map
\[
\cA\longrightarrow \End_\C(\HH)
\]
is surjective. Since both sides have dimension \(|A|^2\), it is an
isomorphism; in particular \(\cA\) is central simple and \(\HH\) is its unique
simple module, up to isomorphism.

The same algebra action gives the projective Weil representation attached to
\(C(A)\). We now construct the projective Weil operators and verify their
covariance. For
\(g\in C(A)\), let \(\HH^g\) be the vector space \(\HH\) with the twisted
\(\cA\)-action
\[
x\cdot_g \xi:=\alpha_g(x)\cdot \xi,
\qquad x\in \cA,\ \xi\in \HH.
\]
Since \(\alpha_g\) is an automorphism, \(\HH^g\) is again simple. The algebra
\(\cA\simeq\End_\C(\HH)\) has, up to isomorphism, a unique simple module, so
there exists an \(\cA\)-module isomorphism
\[
U_g:\HH\xrightarrow{\sim}\HH^g.
\]
By Schur's lemma, \(U_g\) is unique up to a nonzero scalar. Under the
identification \(\cA\simeq\End_\C(\HH)\), the automorphism \(\alpha_g\)
preserves adjoints because it is induced by conjugation with a unitary
representative of the projective Clifford element \(g\). Hence \(U_g^*U_g\)
commutes with \(\cA\), and therefore \(U_g\) may be rescaled to be unitary.
For \(g=(T,\lambda)\), it satisfies
\begin{equation}\label{eq:weil-covariance}
U_gu_vU_g^{-1}=\lambda(v)\,u_{Tv}
\qquad (v\in V_A).
\end{equation}
Both \(U_gU_h\) and \(U_{gh}\) implement \(\alpha_{gh}\). Hence Schur's lemma
gives
\[
U_gU_h=c(g,h)U_{gh}
\]
for some \(c(g,h)\in U(1)\). Thus \(g\mapsto[U_g]\) is a homomorphism
\(C(A)\to\PGL(\HH)\). The scalar ambiguity is irrelevant for the adjoint
action. The Clifford action on
\(\cA\) is strict, while its realization on \(\HH\) is projective. Under
\(\cA\simeq\End_\C(\HH)\), the strict action becomes
\[
\alpha_g=\Ad(U_g).
\]
Thus the projective Weil actions considered in this section have ordinary operator
commutants, independent of the chosen unitary representatives. Their adjoint
actions on \(\cA\) and its tensor powers are genuine
\(C(A)\)-representations.

\subsection{Tensor transport of the twisted group algebra}
\label{subsec:weyl-transport}

Let \(\C[V_A]\) be the commutative group algebra
with basis \(\{\delta_v\}_{v\in V_A}\) and multiplication
\(\delta_v\delta_w=\delta_{v+w}\). It carries a natural action of the affine
symplectic group \(\ASp(A)\) by algebra automorphisms,
\begin{equation}\label{eq:split-action-correct}
(T,\chi)\cdot \delta_v
=
\chi(v)\,\delta_{Tv},
\qquad (T,\chi)\in \ASp(A).
\end{equation}

The word ``commutative'' here refers to the underlying algebra in
\(\Vect\). As an algebra object of
\(\Rep^\omega(\ASp(A))\), it is generally not
commutative: with the symmetry \eqref{eq:omega-braiding},
\[
m\circ c^\omega_{\C[V_A],\C[V_A]}(\delta_v\otimes\delta_w)
=
\omega_A(w,v)\delta_{v+w},
\]
which need not equal \(m(\delta_v\otimes\delta_w)=\delta_{v+w}\). After
transport, the same bicharacter gives the commutation relation in \(\cA\):
\[
u_vu_w=\omega_A(v,w)\,u_wu_v.
\]

\begin{theorem}\label{thm:weyl-transport}
As an algebra object of \(\Rep(C(A))\), the twisted group algebra
\(\cA=\C_{\beta_A}[V_A]\) is isomorphic to \(\cF(\C[V_A])\). Writing \(m\) for
the multiplication of \(\C[V_A]\), the transported multiplication
\(m_{\cF}:=\cF(m)\circ J_{\C[V_A],\C[V_A]}\) is
\[
m_{\cF}(\delta_v\otimes\delta_w)=\beta_A(v,w)\,\delta_{v+w}.
\]
Under \(\delta_v\leftrightarrow u_v\), this is the product
\eqref{eq:twisted-multiplication}; the transported action is then
\eqref{eq:alpha-action}.
\end{theorem}

\begin{proof}
The tensor structure of \(\cF\) is
\[
J(\delta_v\otimes\delta_w)=\beta_A(v,w)\,(\delta_v\otimes\delta_w)
\]
by Theorem~\ref{thm:tensor-equivalence}. Since \(\cF\) is the identity on the
underlying linear maps, one gets
\[
m_{\cF}(\delta_v\otimes\delta_w)
=
m\bigl(\beta_A(v,w)\,\delta_v\otimes\delta_w\bigr)
=
\beta_A(v,w)\,\delta_{v+w}.
\]
This is the multiplication rule of \(\cA\). The same transport turns
\eqref{eq:split-action-correct} into \eqref{eq:alpha-action}. Indeed,
\(\C[V_A]\) has one-dimensional homogeneous component \(\C\delta_v\), and the
linear part sends \(\delta_v\) to \(\delta_{Tv}\). Hence, for
\(g=(T,\lambda_g)\in C(A)\), the functor \(\cF\) gives
\[
g\cdot\delta_v=\lambda_g(v)\,\delta_{Tv},
\]
which is \eqref{eq:alpha-action} under \(\delta_v\leftrightarrow u_v\).
\end{proof}

\subsection{Tensor-power Clifford commutants}
\label{subsec:tensor-power-commutants}

We first consider the commutant of the projective Weil action on tensor powers
of \(\HH\). Choose unitary representatives \(U_g\) for \(g\in C(A)\), and for
\(t\geq1\) put
\begin{equation}\label{eq:def-tensor-power-commutant}
\mathfrak D_t(A)
:=
\{X\in\End_\C(\HH^{\otimes t}):
XU_g^{\otimes t}=U_g^{\otimes t}X\text{ for every }g\in C(A)\}.
\end{equation}
This space is independent of the representatives \(U_g\). We also put
\(\mathfrak D_0(A)=\C\). Under
\(\cA^{\otimes t}\simeq\End_\C(\HH^{\otimes t})\), one has
\begin{equation}\label{eq:tensor-commutant-as-invariants}
\mathfrak D_t(A)=(\cA^{\otimes t})^{C(A)}.
\end{equation}

For \(t\geq1\), define
\[
\Sigma_t:V_A^t\longrightarrow V_A,
\qquad
\Sigma_t(v_1,\ldots,v_t):=v_1+\cdots+v_t,
\]
and put
\begin{equation}\label{eq:zero-sum-tuples}
Z_t(A):=\ker(\Sigma_t).
\end{equation}
For \(\mathbf v=(v_1,\ldots,v_t)\in V_A^t\), write
\begin{equation}\label{eq:kappa-tensor}
\kappa_t(\mathbf v)
:=
\prod_{1\leq i<j\leq t}\beta_A(v_i,v_j),
\qquad
u_{\mathbf v}:=u_{v_1}\otimes\cdots\otimes u_{v_t}.
\end{equation}

For every orbit
\(\mathcal O\in\Sp(V_A)\backslash Z_t(A)\), define
\begin{equation}\label{eq:tensor-orbit-basis}
R_{\mathcal O}
:=
\sum_{\mathbf v\in\mathcal O}
\kappa_t(\mathbf v)^{-1}u_{\mathbf v}.
\end{equation}

\begin{theorem}
\label{thm:tensor-power-commutant}
The operators \(R_{\mathcal O}\), with
\(\mathcal O\in\Sp(V_A)\backslash Z_t(A)\), form an orthogonal basis of
\(\mathfrak D_t(A)\) for the Hilbert--Schmidt inner product. Consequently,
\begin{equation}\label{eq:tensor-commutant-dimension}
\dim\mathfrak D_t(A)
=
|\Sp(V_A)\backslash Z_t(A)|.
\end{equation}
\end{theorem}

\begin{proof}
By Theorem~\ref{thm:weyl-transport} and the strong tensor structure of
\(\cF\),
\[
(\cA^{\otimes t})^{C(A)}
\simeq
(\C[V_A]^{\otimes t})^{\ASp(A)}.
\]
Identify \(\C[V_A]^{\otimes t}\) with \(\C[V_A^t]\). On its standard basis,
\[
(T,\chi)\cdot\delta_{\mathbf v}
=
\chi(\Sigma_t(\mathbf v))\delta_{T\mathbf v}.
\]
Invariance under the normal subgroup \(\widehat{V_A}\) forces
\(\mathbf v\in Z_t(A)\); on this set the remaining action is the diagonal
permutation action of \(\Sp(V_A)\). Hence the orbit sums
\[
\sum_{\mathbf v\in\mathcal O}\delta_{\mathbf v},
\qquad
\mathcal O\in\Sp(V_A)\backslash Z_t(A),
\]
form a basis of the invariant space.

The iterated tensor constraint maps
\[
u_{\mathbf v}\longmapsto\kappa_t(\mathbf v)\delta_{\mathbf v}.
\]
Thus the inverse images of the orbit sums are exactly the operators
\eqref{eq:tensor-orbit-basis}.

The Weyl basis satisfies
\[
\Tr(u_{\mathbf v}^{\dagger}u_{\mathbf w})
=
|A|^t\delta_{\mathbf v,\mathbf w}.
\]
The supports of distinct orbit sums are disjoint, so the operators
\(R_{\mathcal O}\) are orthogonal. The dimension formula follows from the
basis parametrization.
\end{proof}

The same computation gives the norms
\begin{equation}\label{eq:tensor-orbit-orthogonality}
\langle R_{\mathcal O},R_{\mathcal O'}\rangle_{\mathrm{HS}}
=
|A|^t|\mathcal O|\,\delta_{\mathcal O,\mathcal O'}.
\end{equation}

The normalization in \eqref{eq:tensor-orbit-basis} also makes the algebra
structure explicit. For \(\mathbf v\in Z_t(A)\), put
\[
\widetilde u_{\mathbf v}:=\kappa_t(\mathbf v)^{-1}u_{\mathbf v}.
\]
For \(\mathbf v,\mathbf w\in Z_t(A)\), put
\[
\alpha_t(\mathbf v,\mathbf w)
:=
\prod_{1\leq i<j\leq t}\omega_A(w_i,v_j).
\]

\begin{proposition}\label{prop:tensor-commutant-product}
For \(\mathbf v,\mathbf w\in Z_t(A)\), one has
\begin{equation}\label{eq:tensor-phase-product}
\widetilde u_{\mathbf v}\widetilde u_{\mathbf w}
=
\alpha_t(\mathbf v,\mathbf w)\widetilde u_{\mathbf v+\mathbf w}.
\end{equation}
The function \(\alpha_t\) is an \(\Sp(V_A)\)-invariant bicharacter on
\(Z_t(A)\). Consequently,
\begin{equation}\label{eq:tensor-twisted-invariants}
\mathfrak D_t(A)
\simeq
\C_{\alpha_t}[Z_t(A)]^{\Sp(V_A)}
\end{equation}
as algebras.
\end{proposition}

\begin{proof}
Using \eqref{eq:twisted-multiplication}, the coefficient of
\(\widetilde u_{\mathbf v+\mathbf w}\) in
\(\widetilde u_{\mathbf v}\widetilde u_{\mathbf w}\) is
\[
\frac{\kappa_t(\mathbf v+\mathbf w)}
{\kappa_t(\mathbf v)\kappa_t(\mathbf w)}
\prod_{i=1}^t\beta_A(v_i,w_i).
\]
Expanding \(\kappa_t(\mathbf v+\mathbf w)\) and using
\(\Sigma_t(\mathbf v)=\Sigma_t(\mathbf w)=0\) reduces this factor to
\[
\prod_{i<j}
\frac{\beta_A(w_i,v_j)}{\beta_A(v_j,w_i)}
=
\prod_{i<j}\omega_A(w_i,v_j),
\]
which proves \eqref{eq:tensor-phase-product}. The bilinearity and
\(\Sp(V_A)\)-invariance of \(\omega_A\) show that \(\alpha_t\) is an invariant
bicharacter. The normalized Weyl basis therefore realizes the twisted group
algebra \(\C_{\alpha_t}[Z_t(A)]\), and taking \(\Sp(V_A)\)-invariants gives
\eqref{eq:tensor-twisted-invariants}.
\end{proof}

The multiplication in the orbit basis can also be written explicitly. If
\(\mathcal O,\mathcal P,\mathcal Q\) are symplectic orbits in \(Z_t(A)\),
choose \(\mathbf z\in\mathcal Q\) and set
\[
c_{\mathcal O,\mathcal P}^{\mathcal Q}
:=
\sum_{\substack{\mathbf v\in\mathcal O,\ \mathbf w\in\mathcal P\\
\mathbf v+\mathbf w=\mathbf z}}
\alpha_t(\mathbf v,\mathbf w).
\]
The \(\Sp(V_A)\)-invariance of \(\alpha_t\) shows that this scalar is
independent of \(\mathbf z\in\mathcal Q\). Summing
\eqref{eq:tensor-phase-product} over \(\mathcal O\times\mathcal P\) gives
\begin{equation}\label{eq:tensor-orbit-product}
R_{\mathcal O}R_{\mathcal P}
=
\sum_{\mathcal Q}
c_{\mathcal O,\mathcal P}^{\mathcal Q}R_{\mathcal Q}.
\end{equation}

In prime local dimension, tensor-power Clifford commutants have been studied
through finite-geometric and dual-pair methods
\cite{GrossNezamiWalter2021,MontealegreMoraGross2025,BittelEtAl2026}.
Theorem~\ref{thm:tensor-power-commutant} gives the orbit-sum description
uniformly for every finite abelian group \(A\). Its concrete force is greatest
when \(V_A\) is not a vector space over a finite field, so finite-field methods
no longer apply directly.

\subsection{The cyclic case}
\label{subsec:cyclic-tensor-commutants}

We now solve the orbit problem completely when \(A\) is cyclic. We begin with
a prime-power group
\[
A=\Z/p^r\Z,
\qquad
R:=\Z/p^r\Z.
\]
Using the character \(x\mapsto e^{2\pi i x/p^r}\) to identify
\(\widehat{A}\) with \(R\), one has
\[
V_A\simeq R^2,
\qquad
\Sp(V_A)\simeq\SL_2(R).
\]

Fix \(t\geq1\) and put \(m=t-1\). Projection onto the first \(m\)
coordinates identifies \(Z_t(A)\) equivariantly with
\(M_{2\times m}(R)\). Its inverse sends
\[
(v_1,\ldots,v_m)
\longmapsto
(v_1,\ldots,v_m,-v_1-\cdots-v_m).
\]
Writing the vectors \(v_i\in R^2\) as columns, the action of
\(\SL_2(R)\) becomes left multiplication on \(M_{2\times m}(R)\).
For
\[
M=
\begin{pmatrix}
x\\ y
\end{pmatrix}
\in M_{2\times m}(R),
\qquad x,y\in R^m,
\]
put
\begin{equation}\label{eq:row-module-orientation}
L_M:=Rx+Ry\leq R^m,
\qquad
\varepsilon_M:=x\wedge y\in\bigwedge_R^2L_M.
\end{equation}

An \emph{oriented two-generated submodule} of \(R^m\) is a pair
\((L,\varepsilon)\), where \(L\leq R^m\) is generated by at most two
elements and \(\varepsilon\) generates \(\bigwedge_R^2L\). When this
exterior square is zero, its zero element is regarded as its unique
generator.

\begin{theorem}
\label{thm:oriented-submodule-classification}
The assignment \(M\mapsto(L_M,\varepsilon_M)\) induces a bijection
\begin{equation}\label{eq:oriented-submodule-classification}
\SL_2(R)\backslash M_{2\times m}(R)
\xrightarrow{\ \sim\ }
\left\{\text{oriented two-generated submodules of }R^m\right\}.
\end{equation}
\end{theorem}

\begin{proof}
Since \(x,y\) generate \(L_M\), the element \(\varepsilon_M\) generates the
cyclic module \(\bigwedge_R^2L_M\).

Left multiplication does not change the row module, and for
\(S\in\SL_2(R)\) it sends \(x\wedge y\) to
\(\det(S)x\wedge y=x\wedge y\). Thus the assignment is constant on orbits.
It is surjective by choosing a generating pair \(x,y\) of \(L\) whose wedge
is the prescribed generator. Indeed, starting from any generating pair, any
other generator of the cyclic module \(\bigwedge_R^2L\) is obtained by
multiplication by a unit of \(R\), and rescaling one member of the pair by
that unit realizes it.

It remains to prove injectivity. Suppose that \(M\) and \(M'\) have the same
row module \(L\) and the same orientation. If \(L=0\), there is nothing to
prove. If \(L=Rz\) is cyclic, write the rows of \(M\) as \(az,bz\). Since
they generate \(L\), the column \((a,b)^{\mathsf T}\) is unimodular. It can
be completed to a matrix in \(\SL_2(R)\), so every generating pair of \(L\)
lies in the orbit of \((z,0)\).

Assume now that \(L\) needs two generators. Expressing the rows of \(M'\) in
terms of those of \(M\) gives a matrix \(U\in M_2(R)\). Modulo \(p\), the
two pairs are bases of \(L/pL\); hence \(U\in\GL_2(R)\). The equality of the
orientations says
\begin{equation}\label{eq:determinant-fixes-orientation}
(\det U)\varepsilon_M=\varepsilon_M.
\end{equation}
Write
\[
L\simeq \Z/p^a\Z\oplus\Z/p^b\Z,
\qquad r\geq a\geq b\geq1.
\]
Then \(\bigwedge_R^2L\simeq\Z/p^b\Z\), so the units fixing
\(\varepsilon_M\) are precisely those congruent to \(1\) modulo \(p^b\).
The determinant image of the stabilizer in \(\GL_2(R)\) of a generating pair
of \(L\) is exactly this subgroup. To see this explicitly, use invertible row
and column operations to put the generating matrix in Smith normal form
\[
M_0=
\begin{pmatrix}
p^{r-a}&0&0&\cdots&0\\
0&p^{r-b}&0&\cdots&0
\end{pmatrix}.
\]
Left row operations conjugate the stabilizer inside \(\GL_2(R)\), while right
column operations leave it unchanged. Hence neither operation changes its
determinant image. If
\[
H=\begin{pmatrix}h_{11}&h_{12}\\h_{21}&h_{22}\end{pmatrix}
\in\GL_2(R)
\]
satisfies \(HM_0=M_0\), then
\[
h_{11}\equiv1\pmod{p^a},\qquad
h_{21}\equiv0\pmod{p^a},\qquad
h_{12}\equiv0\pmod{p^b},\qquad
h_{22}\equiv1\pmod{p^b}.
\]
Hence \(\det H\equiv1\pmod{p^b}\). Conversely, if
\(q\equiv1\pmod{p^b}\), then
\(\operatorname{diag}(1,q)M_0=M_0\), and this stabilizer element has
determinant \(q\). This proves the asserted description of the determinant
image.

By \eqref{eq:determinant-fixes-orientation}, there is therefore a stabilizer
element \(H\) with \(\det H=(\det U)^{-1}\). The matrix \(UH\) belongs to
\(\SL_2(R)\) and carries \(M\) to \(M'\), proving injectivity.
\end{proof}

Together with Theorem~\ref{thm:tensor-power-commutant}, the equivariant
identification \(Z_t(A)\simeq M_{2\times m}(R)\) gives
\begin{equation}\label{eq:cyclic-dimension-matrix-orbits}
\dim\mathfrak D_t(A)
=
\left|\SL_2(R)\backslash M_{2\times m}(R)\right|.
\end{equation}

For integers \(0\leq j\leq m\), write
\[
\qbinom{m}{j}_p
:=
\prod_{i=0}^{j-1}
\frac{p^{m-i}-1}{p^{j-i}-1}
\]
for the Gaussian binomial coefficient.

For \(t\geq1\) and \(N\geq1\), write
\[
d_t(N):=\dim\mathfrak D_t(\Z/N\Z).
\]
Since \(Z_1(\Z/N\Z)=\{0\}\),
Theorem~\ref{thm:tensor-power-commutant} gives \(d_1(N)=1\).

Fix a prime \(p\) and \(t\geq2\), and put \(m=t-1\). For \(a\geq1\), set
\begin{equation}\label{eq:number-cyclic-submodules}
N_m(a)
:=
p^{(a-1)(m-1)}\qbinom{m}{1}_p.
\end{equation}
For \(m\geq2\) and \(a\geq b\geq1\), set
\begin{equation}\label{eq:number-two-generated-submodules}
N_m(a,b)
:=
\begin{cases}
p^{2(a-1)(m-2)}\qbinom{m}{2}_p,
&a=b,\\[2mm]
p^{(2b-1)(m-2)+(a-b-1)(m-1)}
\qbinom{m-1}{1}_p\qbinom{m}{1}_p,
&a>b.
\end{cases}
\end{equation}

\begin{theorem}
\label{thm:cyclic-commutant-dimension}
For every \(r\geq1\),
\begin{equation}\label{eq:exact-prime-power-dimension}
d_t(p^r)
=
1+
\sum_{a=1}^rN_m(a)
+
\sum_{1\leq b\leq a\leq r}
\varphi(p^b)N_m(a,b),
\end{equation}
where the last sum is omitted when \(m=1\).
\end{theorem}

\begin{proof}
If \(\lambda\) and \(\mu\) are partitions, with conjugate partitions
\(\lambda'\) and \(\mu'\), then Birkhoff's formula
\cite{Birkhoff1935} for the number of subgroups of type \(\mu\) in an abelian
\(p\)-group of type \(\lambda\) is
\begin{equation}\label{eq:birkhoff-subgroup-count}
\alpha_\lambda(\mu;p)
=
p^{\sum_{i\geq1}\mu'_{i+1}(\lambda'_i-\mu'_i)}
\prod_{i\geq1}
\qbinom{\lambda'_i-\mu'_{i+1}}{\mu'_i-\mu'_{i+1}}_p
\end{equation}

For the ambient type \(\lambda=(r^m)\), substituting \(\mu=(a)\) in
\eqref{eq:birkhoff-subgroup-count} gives \(N_m(a)\), the number of cyclic
submodules of \(R^m\) of order \(p^a\). Substituting \(\mu=(a,b)\) gives
\(N_m(a,b)\), the number of submodules of type
\[
\Z/p^a\Z\oplus\Z/p^b\Z,
\qquad a\geq b\geq1,
\]
respectively. Moreover,
\[
\bigwedge_R^2L\simeq\Z/p^b\Z
\]
has \(\varphi(p^b)\) generators. Theorem~\ref{thm:oriented-submodule-classification}
and \eqref{eq:cyclic-dimension-matrix-orbits} therefore give
\eqref{eq:exact-prime-power-dimension}, with the initial term accounting for
the zero submodule.
\end{proof}

If \(N=\prod_{p^r\parallel N}p^r\), the Chinese remainder theorem identifies
\[
\SL_2(\Z/N\Z)
\simeq
\prod_{p^r\parallel N}\SL_2(\Z/p^r\Z)
\]
and decomposes the matrix space into the corresponding product. An orbit of a
product action is a product of orbits. Hence
\begin{equation}\label{eq:cyclic-dimension-multiplicativity}
d_t(N)
=
\prod_{p^r\parallel N}d_t(p^r).
\end{equation}

\begin{remark}\label{rem:cyclic-low-moments}
For \(t=2\), one has \(m=1\), and
\eqref{eq:exact-prime-power-dimension} and
\eqref{eq:cyclic-dimension-multiplicativity} give
\[
d_2(N)=\tau(N),
\]
where \(\tau(N)\) is the number of positive divisors of \(N\). For \(t=3\),
the same equations give
\begin{equation}\label{eq:cyclic-third-moment}
d_3(p^r)
=
p^{r-1}\bigl((r+1)p+r\bigr),
\qquad
d_3(N)
=
\prod_{p^r\parallel N}
p^{r-1}\bigl((r+1)p+r\bigr).
\end{equation}
In particular,
\[
d_3(2)=5,
\qquad
d_3(4)=16,
\qquad
d_3(8)=44,
\qquad
d_3(9)=33.
\]
\end{remark}

\begin{remark}\label{rem:clifford-designs}
Assume \(A\neq0\), and choose unitary representatives \(U_g\) of the
projective Weil representation.
The phase-independent averaged trace moment
\begin{equation}\label{eq:frame-potential-commutant}
\operatorname{FP}_t(C(A))
:=
\frac{1}{|C(A)|}
\sum_{g\in C(A)}|\Tr(U_g)|^{2t}
=
\dim\mathfrak D_t(A).
\end{equation}
The left-hand side is commonly called the \(t\)-th frame potential. To verify
the equality, average the conjugation operators
\[
X\longmapsto
U_g^{\otimes t}X(U_g^*)^{\otimes t}
\]
on \(\End_\C(\HH^{\otimes t})\). Their average is the orthogonal projection
onto \(\mathfrak D_t(A)\), while the trace of the conjugation operator
associated with \(g\) is \(|\Tr(U_g)|^{2t}\).

The uniform collection of projective unitaries
\(\{[U_g]:g\in C(A)\}\) is called a unitary \(t\)-design if this averaging
operator agrees with the corresponding average over \(U(\HH)\) with respect
to normalized Haar measure. Since
\[
\End_{U(\HH)}(\HH^{\otimes t})
\subseteq
\mathfrak D_t(A),
\]
this is equivalent to
\begin{equation}\label{eq:design-commutant-criterion}
\dim\mathfrak D_t(A)
=
\dim\End_{U(\HH)}(\HH^{\otimes t}).
\end{equation}

For \(t=2\), the right-hand side of
\eqref{eq:design-commutant-criterion} has dimension \(2\) by Schur--Weyl
duality. Theorem~\ref{thm:tensor-power-commutant} therefore gives
\[
\{[U_g]:g\in C(A)\}\text{ is a unitary }2\text{-design}
\quad\Longleftrightarrow\quad
A\simeq(\Z/p\Z)^n
\]
for some prime \(p\) and \(n\geq1\). Indeed, the second commutant has dimension
equal to the number of symplectic orbits on \(V_A\), and there are only the
zero and nonzero orbits exactly in the elementary-abelian case
\cite{GraydonSkanesNormanWallman2021}.

Every unitary \(3\)-design is a unitary \(2\)-design, so it remains to consider
\(A=(\Z/p\Z)^n\). In this case, the orbits on
\(Z_3(A)\simeq V_A^2\) give
\[
\dim\mathfrak D_3(A)
=
\begin{cases}
2p+1,&n=1,\\
2p+2,&n\geq2.
\end{cases}
\]
Indeed, the zero pair forms one orbit. The nonzero linearly dependent pairs
are parametrized by \(\mathbb P^1(\F_p)\), the independent pairs with nonzero
symplectic product are parametrized by its value in \(\F_p^\times\), and for
\(n\geq2\) there is one additional orbit of independent isotropic pairs. Thus
the orbit counts are \(1+(p+1)+(p-1)=2p+1\) for \(n=1\), and one larger for
\(n\geq2\). By Schur--Weyl duality,
\[
\dim\End_{U(\HH)}(\HH^{\otimes3})
=
\begin{cases}
5,&\dim\HH=2,\\
6,&\dim\HH\geq3.
\end{cases}
\]
Consequently,
\[
\{[U_g]:g\in C(A)\}\text{ is a unitary }3\text{-design}
\quad\Longleftrightarrow\quad
A\simeq(\Z/2\Z)^n.
\]
This recovers the multiqubit result of \cite{Zhu2017}. Finally, every unitary
\(4\)-design is a unitary \(3\)-design, so only the case
\(A\simeq(\Z/2\Z)^n\) could occur. It is excluded by the known higher-moment
obstruction
\cite{ZhuKuengGrasslGross2016}.
\end{remark}

\subsection{Adjoint-action commutants and tensor transport}

For \(t\ge 0\), let
\[
\mathcal C_t^{\mathrm{cl}}:=\End_{\ASp(A)}(\C[V_A]^{\otimes t}),
\qquad
\mathcal C_t^{\mathrm{q}}:=\End_{C(A)}(\cA^{\otimes t}),
\]
where \(\C[V_A]\) and \(\cA=\C_{\beta_A}[V_A]\) carry the actions
\eqref{eq:split-action-correct} and \eqref{eq:alpha-action}.
Composition and tensor product make \(\mathcal C_*^{\mathrm{cl}}\) and
\(\mathcal C_*^{\mathrm{q}}\) into multiplicative sequences of algebras.

\begin{theorem}
\label{thm:commutant-invariance}
For every finite abelian group \(A\), the tensor isomorphism
\[
\cF:\Rep^\omega(\ASp(A))\xrightarrow{\sim}\Rep(C(A))
\]
induces algebra isomorphisms
\[
\Phi_t:\mathcal C_t^{\mathrm{cl}}\xrightarrow{\sim}\mathcal C_t^{\mathrm{q}},
\qquad t\ge 0,
\]
compatible with the multiplicative structure.
\end{theorem}

\begin{proof}
Since \(\cF\) is an isomorphism of categories, it identifies the endomorphism algebras of any
object and its image. Since it is a strong tensor functor, it also identifies
tensor powers. Finally, Theorem~\ref{thm:weyl-transport} gives an isomorphism
of algebra objects \(\cF(\C[V_A])\cong \cA\). Combining these facts yields the
algebra isomorphisms \(\Phi_t\), and coherence of the tensor structure shows
that they commute with tensor products of endomorphisms.
\end{proof}

Only the tensor structure of \(\cF\) is used in this endomorphism-algebra
identification. The twisted symmetry is needed for the symmetric comparison in
Section~\ref{subsec:twisted-symmetry}, but not for the endomorphism-algebra
isomorphisms in Theorem~\ref{thm:commutant-invariance}.

The adjoint-action commutant is related to the tensor-power
commutant by trace duality.

\begin{proposition}\label{prop:trace-doubling}
For every \(t\geq0\), the trace pairing on \(\cA\simeq\End_\C(\HH)\) induces a
vector-space isomorphism
\begin{equation}\label{eq:trace-doubling}
\mathcal C_t^{\mathrm{q}}
\xrightarrow{\ \sim\ }
\mathfrak D_{2t}(A).
\end{equation}
Under this isomorphism, composition in \(\mathcal C_t^{\mathrm{q}}\) becomes a
contraction product on \(\mathfrak D_{2t}(A)\), rather than the ordinary
multiplication inherited from \(\cA^{\otimes 2t}\).
\end{proposition}

\begin{proof}
There are standard identifications
\[
\End_\C(\cA^{\otimes t})
\simeq
\cA^{\otimes t}\otimes(\cA^{\otimes t})^*
\]
and
\[
\mathcal C_t^{\mathrm{q}}
\simeq
\bigl(\cA^{\otimes t}\otimes(\cA^{\otimes t})^*\bigr)^{C(A)}.
\]
The trace pairing \((x,y)\mapsto\Tr(xy)\) is nondegenerate and
\(C(A)\)-invariant. It therefore identifies
\(\cA^*\) equivariantly with \(\cA\), and gives
\[
\mathcal C_t^{\mathrm{q}}
\simeq
(\cA^{\otimes 2t})^{C(A)}
=
\mathfrak D_{2t}(A).
\]

Composition of endomorphisms contracts the dual and primal tensor factors
through the trace pairing, which gives the contraction product described in
Proposition~\ref{prop:trace-doubling}.
\end{proof}

For \(t\geq1\), define
\begin{equation}\label{eq:Gamma-t}
\Gamma_t(A)
:=
\{(\mathbf x,\mathbf y)\in V_A^t\times V_A^t:
\Sigma_t(\mathbf x)=\Sigma_t(\mathbf y)\}.
\end{equation}

On the Weyl basis one has
\[
\Tr(u_vu_w)=|A|\,\beta_A(v,w)\delta_{v+w,0}.
\]
Thus the element corresponding to the coordinate functional dual to \(u_v\)
is
\[
u_v^\vee
=
\frac{1}{|A|\,\beta_A(v,-v)}u_{-v}.
\]
Consequently, on the Weyl tensor basis the trace identification sends
\[
u_{\mathbf x}\otimes u_{\mathbf y}^*
\longmapsto
|A|^{-t}
\left(\prod_{i=1}^t\beta_A(y_i,-y_i)^{-1}\right)
u_{(x_1,\ldots,x_t,-y_1,\ldots,-y_t)}.
\]
The defining equality in \eqref{eq:Gamma-t} is precisely the zero-sum
condition for
\((x_1,\ldots,x_t,-y_1,\ldots,-y_t)\). Consequently,
\begin{equation}\label{eq:Gamma-zero-sum-identification}
\Gamma_t(A)\xrightarrow{\ \sim\ }Z_{2t}(A),
\qquad
(\mathbf x,\mathbf y)\longmapsto
(x_1,\ldots,x_t,-y_1,\ldots,-y_t),
\end{equation}
is an \(\Sp(V_A)\)-equivariant bijection.

\subsection{The adjoint orbit basis}\label{subsec:classical-commutant-explicit}

We now compute the affine symplectic commutant appearing in
Theorem~\ref{thm:commutant-invariance}. Identify
\(\C[V_A]^{\otimes t}\) with \(\C[V_A^t]\), with basis
\(\{\delta_{\mathbf v}\}_{\mathbf v\in V_A^t}\), where
\(\mathbf v=(v_1,\dots,v_t)\). The diagonal action of
\(\ASp(A)\) is
\begin{equation}\label{eq:diagonal-action}
(T,\chi)\cdot\delta_{\mathbf v}
=
\Bigl(\prod_{i=1}^t\chi(v_i)\Bigr)\delta_{T\mathbf v},
\qquad (T,\chi)\in \ASp(A).
\end{equation}

The equivariant bijection \eqref{eq:Gamma-zero-sum-identification} identifies
the symplectic orbit set of \(\Gamma_t(A)\) with
\(\Sp(V_A)\backslash Z_{2t}(A)\).

\begin{theorem}
\label{thm:adjoint-orbit-basis}
Let \(\operatorname{Orb}_t(A)\) be the set of \(\Sp(V_A)\)-orbits on
\(\Gamma_t(A)\). For \(\alpha\in\operatorname{Orb}_t(A)\), put
\begin{equation}\label{eq:orbit-basis}
e_\alpha:=
\sum_{(\mathbf x,\mathbf y)\in\alpha}
\delta_{\mathbf x}\otimes\delta_{\mathbf y}^*,
\qquad
E_\alpha:=\Phi_t(e_\alpha).
\end{equation}
Then \(\{e_\alpha\}_{\alpha\in\operatorname{Orb}_t(A)}\) and
\(\{E_\alpha\}_{\alpha\in\operatorname{Orb}_t(A)}\) are bases of
\(\mathcal C_t^{\mathrm{cl}}\) and \(\mathcal C_t^{\mathrm{q}}\),
respectively. In particular,
\begin{equation}\label{eq:dim-formula}
\dim\mathcal C_t^{\mathrm{cl}}
=
\dim\mathcal C_t^{\mathrm{q}}
=
\#\operatorname{Orb}_t(A).
\end{equation}
\end{theorem}

\begin{proof}
On the matrix units one has
\[
(T,\chi)\cdot(\delta_{\mathbf x}\otimes\delta_{\mathbf y}^*)
=
\chi\bigl(\Sigma_t(\mathbf x)-\Sigma_t(\mathbf y)\bigr)
\delta_{T\mathbf x}\otimes\delta_{T\mathbf y}^*.
\]
Invariance under \(\widehat{V_A}\) therefore forces
\((\mathbf x,\mathbf y)\in\Gamma_t(A)\). On this set the remaining
\(\Sp(V_A)\)-action simply permutes the matrix units, so its invariant space
has the orbit-sum basis \eqref{eq:orbit-basis}. Applying \(\Phi_t\) from
Theorem~\ref{thm:commutant-invariance} gives the basis
\(\{E_\alpha\}_{\alpha\in\operatorname{Orb}_t(A)}\) and the dimension formula.
\end{proof}

\subsection{The orbit algebra and its structure constants}

The orbit basis also gives an orbit-theoretic description of multiplication. Regard each
\(\alpha\in\operatorname{Orb}_t(A)\) as an \(\Sp(V_A)\)-orbit in
\(\Gamma_t(A)\subseteq V_A^t\times V_A^t\). The affine basis element is
the adjacency operator \(e_\alpha\) in \eqref{eq:orbit-basis}.

For \(\alpha,\beta,\gamma\in\operatorname{Orb}_t(A)\), choose
\((\mathbf x,\mathbf y)\in\gamma\) and set
\begin{equation}\label{eq:orbit-structure-constants}
p_{\alpha,\beta}^{\gamma}
:=
\#\{\mathbf z\in V_A^t:
(\mathbf x,\mathbf z)\in\alpha,\ 
(\mathbf z,\mathbf y)\in\beta\}.
\end{equation}
This number is independent of the chosen pair
\((\mathbf x,\mathbf y)\in\gamma\): if \(S\in\Sp(V_A)\) carries one pair in
\(\gamma\) to another, then \(\mathbf z\mapsto S\mathbf z\) gives the required
bijection.

\begin{proposition}
\label{prop:orbit-structure-constants}
For \(\alpha,\beta\in\operatorname{Orb}_t(A)\), one has
\begin{equation}\label{eq:affine-orbit-multiplication}
e_\alpha e_\beta
=
\sum_{\gamma\in\operatorname{Orb}_t(A)}
p_{\alpha,\beta}^{\gamma}e_\gamma .
\end{equation}
Consequently, for \(E_\alpha=\Phi_t(e_\alpha)\),
\begin{equation}\label{eq:clifford-orbit-multiplication}
E_\alpha E_\beta
=
\sum_{\gamma\in\operatorname{Orb}_t(A)}
p_{\alpha,\beta}^{\gamma}E_\gamma .
\end{equation}
\end{proposition}

\begin{proof}
Let \(M_{\mathbf x,\mathbf y}:=\delta_{\mathbf x}\otimes\delta_{\mathbf y}^*\).
Then
\[
M_{\mathbf x,\mathbf z}M_{\mathbf z',\mathbf y}
=
\begin{cases}
M_{\mathbf x,\mathbf y},&\mathbf z=\mathbf z',\\
0,&\mathbf z\neq\mathbf z'.
\end{cases}
\]
Therefore the coefficient of \(M_{\mathbf x,\mathbf y}\) in
\(e_\alpha e_\beta\) is precisely the number in
\eqref{eq:orbit-structure-constants}. This coefficient is constant on each
\(\Sp(V_A)\)-orbit \(\gamma\). Hence the product has the form
\eqref{eq:affine-orbit-multiplication}. The second formula follows by applying
the algebra isomorphism \(\Phi_t\).
\end{proof}

The multiplication table of the Clifford commutant, written in the
transported orbit basis, is governed by the same intersection numbers
\(p_{\alpha,\beta}^{\gamma}\). The tensor constraint changes the matrices on
the transported side, but not the structure constants. Hence the commutant is an
orbit algebra attached to the action of \(\Sp(V_A)\) on the relation
\(\Gamma_t(A)\).

We finish by recording a comparison with subgroup-supported operators that
appear in related Clifford-commutant constructions. Let
\[
L\le V_A^t\times V_A^t
\]
be a subgroup such that \(L\subseteq\Gamma_t(A)\) and \(L\) is stable under the
diagonal action of \(\Sp(V_A)\). Define
\begin{equation}\label{eq:rL-operator}
r(L):=\sum_{(\mathbf x,\mathbf y)\in L}
\delta_{\mathbf x}\otimes\delta_{\mathbf y}^*
\in \End(\C[V_A^t]).
\end{equation}
Since \(L\) is a union of orbits in \(\operatorname{Orb}_t(A)\), one has
\[
r(L)=\sum_{\alpha\subseteq L}e_\alpha
\in \mathcal C_t^{\mathrm{cl}}.
\]

Recall \(\kappa_t(\mathbf v)\) and \(u_{\mathbf v}\) from
\eqref{eq:kappa-tensor}.
The transported operator \(R(L):=\Phi_t(r(L))\in\mathcal C_t^{\mathrm{q}}\) acts by
\begin{equation}\label{eq:quantum-rL}
R(L)(u_{\mathbf y})
=
\sum_{\mathbf x:\,(\mathbf x,\mathbf y)\in L}
\frac{\kappa_t(\mathbf y)}{\kappa_t(\mathbf x)}\,u_{\mathbf x}.
\end{equation}
Indeed, the iterated tensor constraint sends
\[
u_{\mathbf v}\longmapsto \kappa_t(\mathbf v)\delta_{\mathbf v},
\]
and \(\Phi_t\) is conjugation by this diagonal operator.

When \(L\) is, in addition, maximal isotropic for the difference form
\[
\bigl((\mathbf x,\mathbf y),(\mathbf x',\mathbf y')\bigr)
\longmapsto
\prod_{i=1}^t
\frac{\omega_A(x_i,x_i')}{\omega_A(y_i,y_i')},
\]
then \(r(L)\) is the analogue, in the present equivariant endomorphism algebra,
of the stochastic Lagrangian operators of Gross--Nezami--Walter
\cite[Theorem~2.5]{GrossNezamiWalter2021}. In the usual qudit setting, the
condition \(L\subseteq\Gamma_t(A)\) is the adjoint-action version of the
stochastic condition, while maximal isotropy is the Lagrangian condition. We do
not use these operators to parametrize the commutant. The basis is the orbit
basis \(\{E_\alpha\}\), and subgroup-supported operators are distinguished sums
of its elements.
Proposition~\ref{prop:orbit-structure-constants} records the multiplication
rule for the whole commutant.

\bibliographystyle{amsalpha}
\bibliography{references}

\end{document}